\documentclass[11pt, a4paper]{amsart}

\usepackage{amsthm,amsfonts,amsbsy,amssymb,amsmath,amscd}
\usepackage[all]{xy}
\usepackage[colorlinks=true]{hyperref}
\usepackage{pgf, tikz}

\newcommand{\QQ}{\mathbb{Q}}

\newcommand{\PP}{\mathbb{P}}
\newcommand{\ZZ}{\mathbb{Z}}

\newcommand{\CA}{\mathcal{A}}
\newcommand{\CB}{\mathcal{B}}

\newcommand{\CO}{\mathcal{O}}
\newcommand{\CX}{\mathcal{X}}
\newcommand{\CY}{\mathcal{Y}}
\newcommand{\CZ}{\mathcal{Z}}

\newcommand{\alb}{{\rm alb}}

\newcommand{\Pic}{{\mathrm{Pic}}}

\newcommand{\rounddown}[1]{{\lfloor #1 \rfloor}}

\begin{document}
	
	\title[Slope for fibered varieties over curves]{Fibered varieties over curves with low slope and sharp bounds in dimension three}
	
	\author{Yong Hu}
	\author{Tong Zhang}
	\date{\today}
	
	\address[Y.H.]{School of Mathematics, Korea Institute for Advanced Study, 85 Hoegiro, Dongdaemun-gu, Seoul 02455, South Korea }
	\email{yonghu11@kias.re.kr}
	
	\address[T.Z.]{Department of Mathematics, Shanghai Key Laboratory of PMMP, East China Normal University, 500 Dongchuan Road, Shanghai 200241, People's Republic of China}
	\email{tzhang@math.ecnu.edu.cn, mathtzhang@gmail.com}
	
	\begin{abstract}
		In this paper, we first construct varieties of any dimension $n>2$ fibered over curves with low slopes. These examples violate the conjectural slope inequality of Barja and Stoppino \cite{Barja_Stoppino_1}. 
		
		Led by their conjecture, we focus on finding the lowest possible slope when $n=3$. Based on a characteristic $p > 0$ method, we prove that the sharp lower bound of the slope of fibered $3$-folds over curves is $4/3$, and it occurs only when the general fiber is a $(1, 2)$-surface. Otherwise, the sharp lower bound is $2$. We also obtain a Cornalba-Harris-Xiao type slope inequality for families of surfaces of general type over curves, and it is sharper than all known results with no extra assumptions.
		
		As an application of the slope bound, we deduce a sharp Noether-Severi type inequality that $K_X^3 \ge 2\chi(X, \omega_X)$ for an irregular minimal $3$-fold $X$ of general type not having a $(1,2)$-surface Albanese fibration. It answers a question in \cite{Zhang} and thus completes the full Severi type inequality for irregular $3$-folds of general type.
	\end{abstract}
	
	\maketitle
	
	\tableofcontents

	%%%%%%%%%%%%%%%%%%% Theorems &&&&&&&&&&&&&&&&&&&&&&&&&&&&&&&
	
	\theoremstyle{plain}
	\newtheorem{theorem}{Theorem}[section]
	\newtheorem{lemma}[theorem]{Lemma}
	\newtheorem{coro}[theorem]{Corollary}
	\newtheorem{prop}[theorem]{Proposition}
	\newtheorem{defi}[theorem]{Definition}
	\newtheorem{ques}[theorem]{Question}
	\newtheorem{conj}[theorem]{Conjecture}
	
	\newtheorem*{conj*}{Conjecture}
	
	\theoremstyle{remark}
	\newtheorem{remark}[theorem]{Remark}
	\newtheorem{assumption}[theorem]{Assumption}
	\newtheorem{example}[theorem]{Example}
	
	\numberwithin{equation}{section}
	
	%%%%%%%%%%%%%%%%%%%%%%%%%%%%%%%%%%%%%%%%%%%%%%%%%%%%%%%%%%%%%

	\section{Introduction}
	\subsection{Motivation and the first result}
	
	In \cite[Theorem 1.1]{Cornalba_Harris}, Cornalba and Harris proved that given a surjective flat projective morphism $f: V \to T$ of relative dimension $d$ and a relatively very ample line bundle $L$ on $X$,  the divisor
	$$
	f_*(rc_1(L) - f^*c_1(f_*L))^{d+1}
	$$
	is pseudo-effective on $T$, provided that the linear system associated to $L$ on a general fiber yields a semi-stable embedding in the sense of GIT, where the \emph{embedding} assumption was later weakened by Stoppino \cite{Stoppino}. Here $r = \mbox{rank} f_*L$. Motivated by this result, in a series of recent papers \cite{Barja_Stoppino_1,Barja_Stoppino_2,Barja_Stoppino_3}, Barja and Stoppino introduced and investigated a closely related notion called $f$-positivity. The precise definition is as follows:
	
	\begin{defi} \cite[Definition 3]{Barja_Stoppino_1}
		Given a fibered $n$-dimensional variety $f: X \to B$ over a smooth curve $B$ with general fiber $F$ smooth and a line bundle $L$ on $X$, we say that $L$ is \emph{$f$-positive} if the following inequality holds:
		$$
		L^n \ge n \frac{(L|_F)^{n-1}}{h^0(F, L|_F)} \deg f_* L.
		$$
	\end{defi}
	
	In the case when $F$ is of general type and $L = \omega_{X/B}$, the $f$-positivity of $\omega_{X/B}$ attracts more attentions. The corresponding inequality
	\begin{equation} \label{slopeinequality}
	    K_{X/B}^n \ge n \frac{K_F^{n-1}}{p_g(F)} \deg f_* \omega_{X/B}
	\end{equation}
	is usually called the \emph{slope inequality}, and this has been proved by Cornalba-Harris \cite[Theorem 1.3]{Cornalba_Harris} and Xiao \cite[Theorem 2]{Xiao} when $n=2$.
	
	Based on the Cornalba-Harris method in \cite{Cornalba_Harris}, Barja and Stoppino made the following conjecture:
	
	\begin{conj} \label{conjecture} \cite[Conjecture 1]{Barja_Stoppino_1}
		Let $f: X \to B$ be a fibered $n$-dimensional variety whose relative canonical sheaf $\omega_{X/B}$ is relatively nef and ample on  general fibers, and whose general fibers have sufficiently mild singularities. Then $f$ satisfies the slope inequality \eqref{slopeinequality}.
	\end{conj}
	
	In dimension $n>2$, there are some partial results. In \cite[Theorem 3]{Barja_Stoppino_1}, Barja and Stoppino studied this conjecture when $f_* \omega_{X/B}$ is semi-stable. In \cite{Barja_Stoppino_2,Barja_Stoppino_3}, they further studied this conjecture when $X$ is a relative hypersurface or a relative complete intersection. On the other hand, as is known in \cite{Cornalba_Harris} (see also \cite[\S 4.2]{Barja_Stoppino_1}), Conjecture \ref{conjecture} holds provided that the canonical map of the general fiber is Hilbert-Chow semi-stable.
	
	The first result of this paper is the following theorem.
	
	\begin{theorem} \label{main}
		For any given integers $n > 2$ and $g \ge 0$, there exists a fibration $f: X \to B$ such that
		\begin{itemize}
			\item[(1)] $X$ is smooth of dimension $n$, $B$ is a smooth curve of genus $g$, and the general fiber $F$ of $f$ is smooth;
			
			\item[(2)] the relative dualizing sheaf $\omega_{X/B}$ is ample, and thus $\omega_F$ is ample;
			
			\item[(3)] the following inequality holds:
			$$
			K_{X/B}^n < n \frac{K_F^{n-1}}{p_g(F)} \deg f_* \omega_{X/B}.
			$$
		\end{itemize}
		In particular, \eqref{slopeinequality} does not hold for $f$.
	\end{theorem}
    
    As a result, Conjecture \ref{conjecture} does not hold true. Actually, we construct two counterexamples to the conjecture. In the second one (given in \S \ref{Further remark}), the canonical map of $F$ is finite onto its image. This shows that Conjecture \ref{conjecture} is not even true under this stronger additional assumption.
	
	\subsection{Slope inequality for fibered $3$-folds over curves}
	The above negative answer to Conjecture \ref{conjecture} naturally prompts the following question:
	
	\begin{ques} \label{question}
		Let $f: X \to B$ be a relatively minimal fibration from an $n$-fold $X$ to a curve $B$ such that the general fiber is of general type and that $\deg f_* \omega_{X/B} > 0$. What is the sharp lower bound of $K_{X/B}^n / \deg f_* \omega_{X/B}$?
	\end{ques}
	
	When $n=2$, the above question is known to have an answer. In fact, \eqref{slopeinequality} shows that 
	$$
	\frac{K_{X/B}^2}{\deg f_* \omega_{X/B}} \ge 2.
	$$ 
	Moreover, it is known that if the equality holds and $f$ is not locally trivial, then $f$ is a fibration of genus $2$ curves. However, Question \ref{question} is widely open for any $n > 2$. As is pointed to us by Zuo, such a sharp bound has a deep connection to the Arakelov (in)equality and the canonical class inequality for families of higher dimensional varieties (see \cite{Viehweg_Zuo,Lu_Tan_Zuo} for details). For families of curves, such a connection has been studied even in positive characteristics (see \cite{Lu_Sheng_Zuo} for example).
	
	There are some partial results towards Question \ref{question} in dimension three. For example, Ohno \cite{Ohno} has systematically studied the lower bound of the slope for fibered $3$-folds over curves using Xiao's method. Better bounds using the same method but with extra assumptions have also been obtained by Barja \cite{Barja} afterwards. However, as has been observed by Ohno himself (see \cite[page 645]{Ohno}), when his lower bound is attained, the corresponding fibration \emph{has to be isotrivial}. This particularly means that for a general fibration, results loc. cit. are by no means sharp. 
	
	Conjecture \ref{conjecture} formulated by Barja and Stoppino also suggests a lower bound. By the Noether inequality $K_F^2 \ge 2p_g(F) - 4$ for the general fiber $F$, we know that ${K_F^2}/{p_g(F)} \ge \frac{1}{2}$. Therefore, Conjecture \ref{conjecture} predicts that the lower bound for $n=3$ is at least $\frac{3}{2}$. 
	
	As the second part in this paper, we give a detailed answer to Question \ref{question} when $n=3$. The following theorems are proved in characteristic zero.
	
	\begin{theorem} \label{main2}
		Let $f: X \to B$ be a relatively minimal fibration from a smooth $3$-fold $X$ to a smooth curve $B$ such that the general fiber $F$ is a surface of general type. Then we have the following sharp inequality:
		\begin{equation} \label{slope1}
		K_{X/B}^3 \ge \frac{4}{3} \deg f_* \omega_{X/B}.
		\end{equation}
		Moreover, if $\deg f_* \omega_{X/B} > 0$ and the equality holds, then $(K_F^2, p_g(F)) = (1, 2)$, i.e., $F$ is a $(1, 2)$-surface.
	\end{theorem}
	
	In Section \ref{(1,2) example}, we provide examples for which \eqref{slope1} becomes an equality. The construction of the examples is a generalization over base of arbitrary genus of the construction made by Kobayashi \cite[(3.2) Proposition]{Kobayashi}.  If we assume that the general fiber is not a $(1, 2)$-surface, the slope is \emph{sharply} bounded from below by $2$, which coincides with the result for surface fibrations.
	\begin{theorem} \label{main3}
		Let $f: X \to B$ be a relatively minimal fibration from a $3$-fold $X$ to a smooth curve $B$ such that the general fiber $F$ is a surface of general type. Suppose that $(K_F^2, p_g(F)) \ne (1, 2)$. Then we have the following sharp inequality:
		\begin{equation} \label{slope2}
		K_{X/B}^3 \ge 2 \deg f_* \omega_{X/B}.
		\end{equation}
	\end{theorem}
	
	In Section \ref{(2,3) example}, we provide examples of fibrations of $(2, 3)$-surfaces for which \eqref{slope2} becomes an equality.
	
	In general, we prove the following Cornalba-Harris-Xiao type slope inequality for families of surfaces of general type over curves.
	
	\begin{theorem} \label{main4}
		Let $f: X \to B$ be a relatively minimal fibration from a $3$-fold $X$ to a smooth curve $B$ such that the general fiber $F$ is a surface of general type. Then we have the following inequality:
		\begin{equation} \label{slope3}
		K_{X/B}^3 \ge \left(\frac{4K_F^2}{K_F^2 + 4} \right) \deg f_* \omega_{X/B}.
		\end{equation}
	\end{theorem}
	
	The inequality \eqref{slope3} is a natural generalization of the Cornalba-Harris-Xiao slope inequality \eqref{slopeinequality} for surface fibrations. Recall that \eqref{slopeinequality} for a relatively minimal surface fibration $f: S \to B$ with general fiber $F$ a smooth curve of genus $g \ge 2$ can be reformulated as
	$$
	K_{S/B}^2 \ge \frac{4g-4}{g} \deg f_* \omega_{S/B} = \left( \frac{4 \deg K_F}{\deg K_F + 2} \right) \deg f_* \omega_{S/B}.
	$$
	Actually, it holds in arbitrary characteristic (see \cite{Bost} for example). Thus \eqref{slope3} is of the same type. One notable feature is that \eqref{slope3} also holds in positive characteristics (see Theorem \ref{slopecharp}).
	
	All the above theorems directly follow from Proposition \ref{prop:general case}, \ref{prop:(1,2) surf b0}, \ref{prop:(1,2) surf b1}, \ref{prop:(1,2) surf b2} and \ref{prop:pg3}. In the following, we give an explicit comparison between our results and the previous lower bounds obtained based on Xiao's method.
	
	\begin{table}[!htbp]
		\begin{tabular}{|c|c|c|c|}
			\hline 
			$K_F^2$ & $p_g(F)$ & Previous bound & Our bound \\
			\hline 
			$=1$ & $\le 2$ & $s \ge 1$ & $s \ge \frac{4}{3}$ \\
			\hline
			$=2$ & $\le 3$ & $s \ge \frac{4}{3}$ & $s \ge 2$ \\
			\hline
			$=3$ & $\le 3$ &$s \ge \frac{4}{3}$ & $s \ge 2$ \\
			\hline
			$\ge 4$ & $\le \frac{K^2_F}{2} + 2 $& $s \ge \frac{4(p_g(F)-2)}{p_g(F)}$ & $s \ge \frac{4K_F^2}{K_F^2 + 4}$ \\
			\hline
		\end{tabular}
		\smallskip
		\caption{Lower bound of the slope $s: = \frac{K_{X/B}^3}{\deg f_* \omega_{X/B}}$ } \label{table}
	\end{table}
	
	The third column ``Previous bound" in \textsc{Table} \ref{table} are straightforward computations from \cite[Proposition 2.1, 2.6]{Ohno}. By the Noether inequality, it is easy to check that
	$$
	\frac{4K_F^2}{K_F^2 + 4} \ge \frac{4(p_g(F)-2)}{p_g(F)},
	$$
	and the equality holds if and only if $K_F^2 = 2p_g(F) - 4$. Therefore,  results in this paper have improved all known results that
	have no extra requirements, with only one possible exception when $F$ is an even Horikawa surface with $p_g(F) \ge 4$ for which two results coincide.
	
	The lower bound of the slope can be further improved if extra assumptions on fibers are added. This has been studied in details by Barja in \cite{Barja}. Among others, he showed that for $p_g(F) \gg 0$, the lower bound of the slope is $9 - O(\frac{1}{p_g(F)})$ under the additional assumptions that the canonical map of $F$ is generically finite and that $F$ has no pencils of curves with gonality less than five [loc. cit., Theorem 0.3]. This bound is not implied by Theorem \ref{main4}. Notice that the above extra assumptions are essential and cannot be removed. Otherwise the lower bound of the slope is no bigger than $\frac{26}{5} - O(\frac{1}{p_g(F)})$. See the example in \S \ref{Further remark}.
	
	\subsection{An application to the geography of irregular $3$-folds}
	For an irregular  minimal surface $S$ of general type, it is well known, at least due to Bombieri \cite[Lemma 14]{Bombieri} (see \cite[Proposition 2.3.2]{Lopes_Pardini} also), that we have the following Noether type inequality
	$$
	K_S^2 \ge 2 \chi(S, \omega_S).
	$$
	Moreover, this inequality is sharp. If the equality holds, by a result of Debarre \cite[Th\'eor\`eme 6.1]{Debarre} that $K_S^2 \ge 2h^0(S, K_S)$, we know that $h^1(S, \CO_S) = 1$. This is one of the most fundamental results in the geography of irregular surfaces of general type. We refer the reader to the survey \cite{Lopes_Pardini} for details regarding this topic.
	
	Recently, the geography of $3$-folds of general type draws lots of attentions. For example, for Gorenstein minimal $3$-folds of general type, the sharp Noether inequality between $K^3$ and $p_g$ and the Noether type inequality between $K^3$ and $\chi$ have been established by Chen and Chen \cite{Chen_Chen} and the first named author \cite{Hu}, respectively. For irregular minimal $3$-folds $X$ of general type, let $d$ denote the Albanese dimension of $X$. The Severi type inequality that 
	\begin{equation} \label{severi}
		K_X^3 \ge 2d! \chi(X, \omega_X)
	\end{equation}
	for $d=2$ and $d=3$ has also been established in \cite[Corollary B, Example 2.4]{Barja_Severi} and \cite[Theorem 1.3]{Zhang} (see \cite{Barja_Severi,
		Zhang_small_volume} for more general results). However, the existence of above inequality for $d=1$, posted as a question in \cite[Conjecture 1.4]{Zhang}, is still very unclear. 
	
	As an application of Theorem \ref{main3}, we verify the remaining unknown case of the Severi type inequality \eqref{severi} by setting up the following theorem.
	
	\begin{theorem} \label{main5}
		Let $X$ be an irregular minimal $3$-fold of general type over an algebraically closed field of characteristic zero. Suppose that the Albanese fiber of $X$ is not a $(1, 2)$-surface. Then we have the following sharp inequality:
		$$
		K^3_X \ge 2 \chi(X, \omega_X).
		$$
		Moreover, if the equality holds, then $h^1(X, \CO_X) = 1$. 	
	\end{theorem}

    It is easy to see that Theorem \ref{main5} is almost compatible with the Noether type result that we have mentioned before for surfaces. Thus we call it a \emph{Noether-Severi type inequality}. The only difference from the surface case is that, we have to exclude the case when the Albanese fiber of $X$ is a $(1, 2)$-surface, because there does exist a counterexample in this case for which $K_X^3 = \frac{4}{3} \chi(X, \omega_X) < 2 \chi(X, \omega_X)$. See Remark \ref{remark: (1,2) surface}. It is worth mentioning here that in \cite{Barja_Severi,Barja_Pardini_Stoppino}, a very similar set of inequalities are called Clifford-Severi inequalities.
    
    To finish this subsection, we summarize the full Severi type inequality for irregular $3$-folds of general type as follows.
    
    \begin{coro}  [Full Severi type inequality in dimension three] \label{full_severi}
    	Let $X$ be an irregular and minimal $3$-fold of general type over an algebraically closed field of characteristic zero. Then the following inequality
    	$$
    	K_X^3 \ge 2(\alb\dim(X))! \chi(X, \omega_X)
    	$$
    	holds, with the only exception when $\alb\dim(X) = 1$ and a general Albanese fiber is a $(1, 2)$-surface.
    \end{coro}

    We would like to mention that the full Severi type inequality \eqref{severi} can be deduced from \cite[Example 2.4]{Barja_Severi} by Barja if $h^0_{\alb_X}(X, K_X) \ge \chi(X, \omega_{X})$ (see the notation loc. cit.). However, this assumption is not always satisfied when $\alb\dim(X) = 1$. See Remark \ref{continuous rank} for example.
    
    It is also worth mentioning that the above full Severi type inequality holds for irregular surfaces of general type in arbitrary characteristic. See \cite{Lopes_Pardini} for characteristic zero and \cite{Yuan_Zhang} for positive characteristics. Thus it might be natural to wonder whether Corollary \ref{full_severi} holds also in positive characteristics.

	\subsection{Proof of sharp bounds: a characteristic $p > 0$ method}
	The rough idea of proofs of all above results can be summarized simply as follows: by reduction mod $p$, it suffices to prove the corresponding slope inequality for $f: X \to B$ over a field of characteristic $p > 0$. Then we study the $e^{\mathrm{th}}$ Frobenius base change $f_e: X_e \to B$ of $f$ and prove a \emph{slope-like inequality} for $f_e$ but with error terms. Finally, the error term is removed by letting $e \to \infty$. Nevertheless, to realize this very rough idea, many other ideas and techniques are needed.
	
	\subsubsection{Proof of Theorem \ref{main2}} To illustrate our method in a concrete way, we start from a sketch of the proof of Theorem \ref{main2}.
	
	In fact, the main difficulty lies in the case when $K_F^2 = 1$. Thus $1 \le p_g(F) \le 2$. The more difficult case is when $F$ is a $(1, 2)$-surface. To prove the sharp bound in \eqref{slope1}, our method in this paper employs a characteristic $p>0$ argument. More precisely, let $k$ be a field of characteristic $p>0$, and let $f: X \to B$ be a $3$-fold fibration defined over $k$ with general fiber $F$ a $(1, 2)$-surface such that the fibration is ``specialized from characteristic zero". We prove that
	\begin{equation} \label{idea}
	(K_{X/B} + f^*A)^3 \ge \frac{4}{3} \deg f_* \omega_{X/B}
	\end{equation}
	holds for an ample $\QQ$-divisor $A$ on $B$ with $\deg A$ arbitrarily small. This just implies \eqref{slope1} by letting $\deg A \to 0$.
	
	To prove \eqref{idea}, we use in an extensive way the ``distinguished" section $\Gamma_0$ of $f$ which comes from the base point of $|K_F|$. By carefully tracing the behavior of $\Gamma_0$ under Frobenius base changes and checking the difference of canonical divisors under normalizations and resolutions of singularities, we establish a \emph{slope-like comparison} between $(K_{X_e/B} + f_e^*A)^3$ and $h^0(X_e, \rounddown{K_{X_e/B}})$ similar to \eqref{idea}, and \eqref{idea} is a limit version of this comparison after taking $e \to \infty$. 
	
	To deal with the $(1, 1)$-surface fiber case, we still apply the above characteristic $p > 0$ argument but to a different setting. Via a relation between $\deg f_* \omega_{X/B}$ and $\deg f_* \omega_{X/B}^{[2]}$, we discover that to prove \eqref{slope2} for $(1, 1)$-surface fibrations, it suffices to prove that
	$$
	K_{X/B}^3 \ge \frac{1}{2} \deg f_* \omega_{X/B}^{[2]},
	$$
	which follows from the inequality
	\begin{equation} \label{idea2}
	(K_{X/B} + f^*A)^3 \ge \frac{1}{2} \deg f_* \omega_{X/B}^{[2]}
	\end{equation}
	over a field $k$ of characteristic $p > 0$. This is the key observation in this case. With this observation, the \emph{slope-like inequality} corresponding to \eqref{idea2} can be obtained by studying the map induced by $|\rounddown{2K_{X_e/B} + 2f_e^*A}|$. Notice that some delicate geometry of $(1, 1)$-surfaces specialized from characteristic zero plays an essential role here.
	
	\subsubsection{Proof of Theorem \ref{main3} and \ref{main4}} Now we move to the proof of Theorem \ref{main3}. Here we also need to consider the case when $p_g(F) = 3$ and $K_F^2=2$, or $3$. However, similarly to the $(1, 1)$-surface fiber case, we study the map induced by $|\rounddown{K_{X_e/B}+f^*A}|$, and this directly yields
	a \emph{slope-like} inequality which helps to complete the proof of Theorem \ref{main3}. Moreover, the proof of Theorem \ref{main4} is roughly the same.
	
	\subsubsection{Comparison with Xiao's method}
	In fact, it would be very interesting, at least to us, to find characteristic zero proofs of Theorem \ref{main2}, \ref{main3} and \ref{main4}. 
	
	There are some known candidates which have been used throughout the past decades to study the slope problem: one is the GIT method \cite{Cornalba_Harris} or the Chow stability method by Bost \cite{Bost}\footnote{Bost's method works in any characteristic.}, which works when the canonical map of the general fiber is finite onto its image and semi-stable (the semi-stability assumption cannot be dropped); one is the method by Moriwaki \cite{Moriwaki} using the Bogomolov instability theorem over surfaces, which seems very difficult to be generalized to higher dimensions (see \cite[\S 3.4]{Barja_Stoppino_1} for a detailed discussion); the other is Xiao's method \cite{Xiao}, which has been widely applied further to various slope problems as well as other problems. See \cite{Ohno,Konno,Barja,Barja_Stoppino_1} for example.
	
	From \textsc{Table} \ref{table}, we have seen that for $3$-fold fibrations (with no extra assumptions on fibers), our characteristic $p$ method yields better bounds than those obtained via Xiao's method. This is somewhat surprising, because for surface fibrations, it is Xiao's method that leads to the sharp lower bound. It looks to us that starting from $n=3$, results via the two methods bifurcate from each other. One possible reason is when the fiber is no longer a curve, its geometry becomes much more complicated so that the estimate via the \emph{original} Xiao's method is not that accurate as for curve fibers.

	\subsection*{Notation and conventions} In this paper, if not specially mentioned, we always work over an algebraically closed field of characteristic zero. The characteristic $p > 0$ arguments appear only in Section \ref{12proof}--\ref{remainingproof}. All varieties are assumed to be projective. We use $\sim$ to denote the linear equivalence and use $\equiv$ to denote the numerical equivalence.
	
	\subsubsection*{Fibration} A fibration always means a surjective morphism with connected fibers. Let $f: X \to B$ be a fibration from $X$ to a curve $B$. We say that $f$ is relatively minimal, if $X$ is projective and normal, with at worst terminal singularities, and the divisor $K_X$ is $f$-nef. Notice that Ohno \cite[Theorem 1.4]{Ohno} has proved that this implies that $K_{X/B}$ is nef in characteristic zero.
	
	\subsubsection*{Abbreviation} In this paper, we always use the notion of \emph{$(a, b)$-surfaces}. By an  $(a, b)$-surface, we mean a minimal surface of general type with $K^2=a$ and $p_g=b$.
	
	\subsection*{Acknowledgement} Y.H. would like to thank Professors JongHae Keum and Jun-Muk Hwang for their generous support during  his stay at KIAS. He also would like to thank Professor Jungkai Chen for his interest in this paper. T.Z. would like to thank Professor Meng Chen for communications since 2013 about the full Severi type inequality in dimension three. Both of the authors would like to thank Professors Miguel \'Angle Barja and Lidia Stoppino for their interest in this paper, as well as Professor Kang Zuo for his very enlightening comments and suggestions regarding this paper. Last but not least, they would like to thank sincerely the anonymous referee for his(her) valuable comments and insights which help to improve the presentation of the paper dramatically.
	
	The work of Y.H. is supported by National Researcher Program of National Research Foundation of Korea (Grant
	No.~2010-0020413). The work of T.Z. is supported by a Science and Technology Commission of Shanghai Municipality (STCSM) Grant No. 18dz2271000. T.Z. was also funded by a Leverhulme Trust Research Project Grant ECF-2016-269 when he started working on this problem at Durham University.
	
	\section{Proof of Theorem \ref{main}} \label{proof}
	
	In this section, we will prove Theorem \ref{main} by a very explicit construction.
	
	\subsection{The construction} Let $g \ge 0$ and $n > 2$ be the two given integers. We start from the following basic data:
	\begin{itemize}
		\item [(i)] Let $Z = A \times B$ be the product of $A$ and $B$, where $B$ is a smooth curve of genus $g$ and $A$ is an abelian variety of dimension $n-2$.
		
		\item [(ii)] Let $p_A$ and $p_B$ be the two natural projections of $Z$. We choose an ample divisor $D_A$ on $A$ and an ample divisor $D_B$ on $B$. Then the divisor
		$$
		D := p_A^*D_A + p_B^*D_B
		$$
		is ample on $Z$. We may further assume that the linear system $|2D|$ is base point free.
		
		\item [(iii)] Let
		$$
		p: Y=\PP(\CO_Z \oplus \CO_Z(-2D)) \to Z
		$$
		be the $\PP^1$-bundle over $Z$. Let $H$ be the section associated to $\CO_Y(1)$ which corresponds to the natural projection $\CO_Z \oplus \CO_Z(-2D) \to \CO_Z(-2D)$.
	\end{itemize}

	Here the first observation is that the linear system $|H + 2p^*D|$ is base point free, and $\CO_H(H + 2p^*D) = \CO_H$. Therefore, we can find a smooth divisor $H' \in |5H + 10p^*D|$ such that $H'$ and $H$ do not intersect with each other. Notice that $H+H'$ is linear equivalent to an even divisor $6H + 10p^*D$. This induces a double cover $\pi: X \to Y$ branched along $H$ and $H'$. In particular, $X$ is smooth. To sum up, we have the following commutative diagram:
	$$
	\xymatrix{
		X \ar[r]^{\pi} \ar[dd]_{f} & Y \ar[d]^p \ar[ldd]_{\sigma} & \\
		& Z \ar[ld]^{p_B} \ar[rd]_{p_A} & \\
		B & & A
	}
	$$
	In the above diagram, we denote by $f: X \to B$ (resp. $\sigma: Y \to B$) the induced fibration from $X$ (resp. $Y$) to $B$.
	
	\subsection{Calculations} \label{calculation} Denote by $F$ a general fiber of $f$ and by $\Sigma$ a general fiber of $\sigma$. Then we have the following proposition.
	
	\begin{prop} \label{fiberdata}
		The variety $F$ is minimal of general type with $K_F^{n-1} = (3^{n-1}-1)D_A^{n-2}$ and $p_g(F) = \frac{ (3^{n-2} + 1)}{(n-2)!} D_A^{n-2}$.
	\end{prop}
	
	\begin{proof}
		It follows from the previous construction that
		$$
		\Sigma \cong \PP(\CO_A \oplus \CO_A(-2D_A)).
		$$
		By abuse of notation, we still denote by $p: \Sigma \to A$ the $\PP^1$-bundle structure on $\Sigma$ and denote by $H_A$ the only effective divisor associated to $\CO_{\Sigma}(1)$ which corresponds to the natural projection $\CO_\Sigma \oplus \CO_\Sigma(-2D_A) \to \CO_\Sigma(-2D_A)$. Then the restriction of $\pi$ on $F$ is just a double cover over $\Sigma$ branched along $(H+H')|_{\Sigma} \sim 6H_A + 10p^*D_A$.
		
		By the double cover formula and the fact that $K_A \sim 0$, we have
		$$
		K_F \sim \pi^*(K_\Sigma + 3H_A + 5p^*D_A) = \pi^*(H_A + 3p^*D_A).
		$$
		Notice that $H_A + 3p^*D_A$ is ample on $\Sigma$. We conclude that $K_F$ is ample. Thus $F$ is minimal of general type. Also from the above formula, we have
		$$
		p_g(F) = h^0(F, K_F) = h^0(\Sigma, H_A + 3p^*D_A).
		$$
		Notice that $p_*\CO_{\Sigma}(H_A + 3p^*D_A) = \CO_A(D_A) \oplus \CO_A(3D_A)$. Therefore, from the Riemann-Roch formula on abelian varieties as well as the Kodaira vanishing theorem, we obtain that
		$$
		p_g(F) = h^0(A, D_A) + h^0(A, 3D_A) = \frac{ (3^{n-2} + 1)}{(n-2)!} D_A^{n-2}.
		$$
		In the meantime, we can calculate that
		$$
		K_F^{n-1} = 2(H_A + 3p^*D_A)^{n-1} = (3^{n-1}-1)D_A^{n-2}.
		$$
		Thus the proof is completed.
	\end{proof}
	
	\begin{prop} \label{Kn}
		We have $K_{X/B}^n = (3^{n}-1)(n-1)(\deg D_B) D_A^{n-2}$.
	\end{prop}
	
	\begin{proof}
		This is similar to the proof of Proposition \ref{fiberdata}. Notice that we have
		$$
		K_{X/B} \sim \pi^*(K_{Y/B} + 3H + 5p^*D) = \pi^*(H+3p^*D).
		$$
		We can similarly deduce that
		$$
		K_{X/B}^n = (3^n-1)D^{n-1}.
		$$
		Thus the result follows from the fact that $D^{n-1}=(n-1)(\deg D_B) D_A^{n-2}$. Therefore, the proof is completed.
	\end{proof}
	
	\begin{prop} \label{degree}
		We have
		$$
		f_* \omega_{X/B} = \CO_B(D_B)^{\oplus h^0(A, D_A)} \oplus \CO_B(3D_B)^{\oplus h^0(A, 3D_A)}.
		$$
		In particular, $\deg f_*\omega_{X/B} = \frac{(3^{n-1} + 1)}{(n-2)!}(\deg D_B)D_A^{n-2}$.
	\end{prop}
	
	\begin{proof}
		To calculate $f_* \omega_{X/B}$, we just use the fact that $f = p_B \circ p \circ \pi$. First, from the double cover formula above, we know that
		$$
		\pi_* \omega_{X/B} = \omega_{Y/B} \oplus  \CO_Y(H + 3p^*D).
		$$
		Thus by the projection formula and by the fact that $Y$ is a $\PP^1$-bundle, we obtain
		$$
		p_* (\pi_* \omega_{X/B} ) = \CO_Z(D) \oplus \CO_Z(3D).
		$$
		Using the construction of $D$, we deduce that
		\begin{eqnarray*}
			f_*\omega_{X/B} & = & {p_B}_*\CO_Z(D) \oplus {p_B}_* \CO_Z(3D) \\
			& = & \CO_B(D_B)^{\oplus h^0(A, D_A)} \oplus \CO_B(3D_B)^{\oplus h^0(A, 3D_A)}.
		\end{eqnarray*}
		The degree of $f_*\omega_{X/B}$ is just a direct calculation from the Riemann-Roch formula again. Thus the proof is completed.
	\end{proof}
	
	\subsection{Conclusion} From the propositions in the previous subsection, we simply have
	$$
	\frac{K_F^{n-1}}{p_g(F)} = \frac{(3^{n-1}-1)(n-2)!}{3^{n-2} + 1}
	$$
	and
	$$
	\frac{K_{X/B}^n}{\deg f_* \omega_{X/B}} = \frac{(3^n-1)(n-1)!}{3^{n-1} + 1}.
	$$
	It is straightforward to check that
	$$
	n \frac{K_F^{n-1}}{p_g(F)} > \frac{K_{X/B}^n}{\deg f_* \omega_{X/B}} \Longleftrightarrow (n-1)(3^n-3^{n-2}) - 3^{2n-2} + 1 < 0,
	$$
	and the inequality on the right hand side holds for any $n > 2$. Thus the proof of Theorem \ref{main} is completed.
	
	\subsection{Further remark} \label{Further remark} In fact, in the previous construction, the variety $A$ does not have to be an abelian variety. The reason we choose an abelian variety is just to simplify the calculation a bit. Here we give another example by assuming that $A = \PP^1$. Thus $n=3$. Since the calculation here is very similar to the above, we just give a sketch.
	
	We apply the same notation as before, and further assume that $\deg D_A > 2$. Then
	$$
	K_F = \pi^*(K_\Sigma + 3H_A + 5p^*D_A) = \pi^*(H_A + 3p^*D_A + p^*K_A).
	$$
	Since $\deg D_A > 2$, the divisor $H_A + 3p^*D_A + p^*K_A$ is still ample. So is $K_F$. Moreover, we have
	$$
	K_F^2 = 2(H_A + 3p^*D_A + p^*K_A)^2 = 8 (\deg D_A - 1).
	$$
	We also deduce that
	$$
	p_g(F) = h^0(A, D_A + K_A) + h^0(A, 3D_A + K_A) = 4 \deg D_A - 2.
	$$
	Notice that $F$ is a Horikawa surface. Thus its canonical map is finite onto its image of degree two. 
	
	Now we turn to the invariant of $K_{X/B}$. We can similarly deduce that
	$$
	K_{X/B} = \pi^*(H + 3p^*D + p^*K_{Z/B}).
	$$
	Thus we can calculate that
	$$
	K_{X/B}^3 = 2 (H+3p^*D + p^*K_{Z/B})^3 = 4(\deg D_B) (13 \deg D_A - 12).
	$$
	Furthermore, we know that
	\begin{eqnarray*}
		f_* \omega_{X/B} & = & {p_B}_*\CO_Z(D + K_{Z/B}) \oplus {p_B}_*\CO_Z(3D + K_{Z/B}) \\
		& = & \CO_B(D_B)^{\oplus h^0(A, D_A+K_A)} \oplus \CO_B(3D_B)^{\oplus h^0(A, 3D_A+K_A)}.
	\end{eqnarray*}
	Therefore, we deduce that
	$$
	\deg f_* \omega_{X/B} = (\deg D_B)(10 \deg D_A - 4).
	$$
	As a result, we can easily verify that in this case, the inequality \eqref{slopeinequality} fails by checking that
	$$
	\frac{K_{X/B}^3}{\deg f_* \omega_{X/B}} = \frac{4(13\deg D_A - 12)}{10\deg D_A - 4} < \frac{24(\deg D_A - 1)}{4\deg D_A-2} = \frac{3K_F^2}{p_g(F)}.
	$$
	In fact, \eqref{slopeinequality} also fails when $\deg D_A = 2$. However, in this case, $K_F$ is nef and big, but not ample.
	
	For general $n > 3$, if we take $A = \PP^{n-2}$, similar counterexamples to \eqref{slopeinequality} can also be produced. We leave the calculation to the interested reader.
	
	\begin{remark} \label{continuous rank}
		Go back to the above example when $n=3$. Notice that $h^1(F, \CO_F) = 0$. Since $R^1 f_*\omega_X$ is torsion free, we deduce that $R^1 f_*\omega_X = 0$. From the Leray spectral sequence, we deduce that
		\begin{eqnarray*}
			h^1(X, \omega_X) & = & h^1(B, f_* \omega_X) = h^0(B, (f_* \omega_{X/B})^\vee) = 0; \\
			h^2(X, \omega_X) & = & h^0(B, R^2 f_* \omega_X) = h^0(B, \omega_B).
		\end{eqnarray*}
		Thus it follows that
		$$
		h^0(X, \omega_X) = \chi(X, \omega_X) - \chi(B, \omega_B).
		$$
		In particular, if $g(B) \ge 2$, for general $\alpha \in \Pic^0(X)$, we have 
		$$
		h^0(X, \omega_X \otimes \alpha) < \chi(X, \omega_X).
		$$
	\end{remark}
	
	\section{A $3$-fold example with the sharp slope bound} \label{(1,2) example}
	
	In this section, we mainly focus on $3$-fold fibrations over curves.
	
	Recall that the Noether inequality asserts that $K_F^2 \ge 2p_g(F) - 4$ if $F$ is a minimal surface of general type. In particular, one can easily check that \eqref{slopeinequality} plus the above Noether inequality would imply that
	$$
	K_{X/B}^3 \ge \frac{3}{2} \deg f_* \omega_{X/B}
	$$
	when the general fiber $F$ of $f$ is smooth and $p_g(F)>0$. In particular, the smallest slope should be achieved when $K^2_F=1$ and $p_g(F)=2$. However, in this section we prove that
	\begin{prop} \label{example}
		For any $g \ge 0$, there is a $3$-fold fibration $f: X \to B$ to a curve $B$ of genus $g$ which satisfies the assumption in Conjecture \ref{conjecture} but
		$$
		K_{X/B}^3 = \frac{4}{3} \deg f_* \omega_{X/B}.
		$$
	\end{prop}
	As a result, this proposition disproves Conjecture \ref{conjecture} when $n=3$. The whole section is devoted to the proof of Proposition \ref{example}. The idea here comes from Kobayashi \cite{Kobayashi}.
	
	\subsection{Construction and calculation}
	Fix a curve $B$ of genus $g \ge 0$ and an effective divisor $D_B$ of degree $e \ge 3$ on $B$. Let 
	$$
	p_B: S = \PP(\CO_B \oplus \CO_B(-D_B)) \to B
	$$ 
	be the $\PP^1$-bundle over $B$, Let $B'$ be the negative section associated to $\CO_S(1)$ corresponding to the projection $\CO_B \oplus \CO_B(-D_B) \to \CO_B(-D_B)$, and let $C$ be a fiber of $p_B$. Denote $D = B' + eC$, and let 
	$$
	p: Y = \PP(\CO_S \oplus O_S(-2D)) \to S
	$$ 
	be the $\PP^1$-bundle over $S$. Let $H$ be the section associated to $\CO_Y(1)$ which corresponds to the projection $\CO_S \oplus \CO_S(-2D) \to \CO_S(-2D)$. Similarly to the previous section, we can find a smooth divisor $H' \in |5H + 10p^*D|$ such that $H$ and $H'$ are disjoint. Let 
	$$
	\pi: X' \to Y
	$$ 
	be the double cover branched along $H$ and $H'$. The same as in \S \ref{calculation}, we have
	\begin{eqnarray*}
		K_{X'} & = & \pi^*(K_Y + 3H + 5 p^*D) = \pi^*(p^*(K_S + 3D) + H) \\
		& = & \pi^*(p^*(B'+(2g-2 + 2e)C) + H).
	\end{eqnarray*}
	For simplicity, we denote $N = B'+(2g-2 + 2e)C$. Notice that $N$ is always ample for any $g$. We also write $\pi^*H = 2\Sigma$. However, the difference from the previous construction is that: $K_{X'}$ is not nef. In the following, the variety $X$ we will consider is the minimal model of $X'$. This is in fact the key point in \cite{Kobayashi}.
	
	Write $f': X' \to B$ and $p': X' \to S$ to be the induced fibration from $X'$ to $B$ and $S$, respectively. Let $M = K_{X'} - \Sigma = \Sigma + p'^*N$.
	\begin{lemma} \label{lemma1}
		The divisors $M$ and $M - f'^*K_B$ are both nef and big. Moreover, we have
		$$
		(M - f'^*K_B)^3 = 4e
		$$
	\end{lemma}
	
	\begin{proof}
		Notice that $\Sigma \simeq S$. Under this isomorphism, we have
		\begin{equation} \label{restriction}
		M|_\Sigma = \Sigma|_\Sigma + (p'^*N)|_\Sigma \sim (N - D)|_S = (2g-2+e) C.
		\end{equation}
		Thus for any curve $\Gamma$ on $X$, if $\Gamma \subset \Sigma$, then $(M \cdot \Gamma) \ge 0$. If not, then $(M\cdot \Gamma)\ge 0$. Using the fact that $N$ is ample, we conclude that $(M \cdot \Gamma) \ge (p'^*N \cdot \Gamma) \ge 0$. A similar argument gives the nefness of $M-f'^*K_B$.
		
		On the other hand, we have
		\begin{eqnarray*}
			M^3 & = & \Sigma^3 + 3 (\Sigma^2 \cdot (p'^*N)) + 3 (\Sigma\cdot (p'^*N)^2) \\
			& = & D^2 - 3(D\cdot N) + 3N^2 \\
			& = & 4e + (6g-6).
		\end{eqnarray*}
		Thus $M$ is also big.
		
		Finally, since $M -f'^*K_B = \Sigma + p'^*(B' + 2eC)$, a similar calculation shows that $(M -f'^*K_B)^3 = 4e$.
	\end{proof}
	
	\begin{lemma}
		The divisor $3M-K_{X'}$ is nef and big.
	\end{lemma}
	
	\begin{proof}
		Since $3M-K_{X'} = p'^*N + M$ and both $N$ and $M$ are nef and big, the conclusion follows.
	\end{proof}
	
	\begin{lemma}
		For any irreducible curve $\Gamma$ on $X$, $(M \cdot \Gamma) = 0$ if and only if $\Gamma$ is a ruling on $\Sigma$.
	\end{lemma}
	
	\begin{proof}
		The ``only if" part follows from \eqref{restriction}. Now suppose $(M \cdot \Gamma) = 0$. Then $\Gamma$ can not be vertical with respect to $p'$, otherwise $(M \cdot \Gamma) = \frac{1}{2}((\pi^*H) \cdot \Gamma) > 0$. Thus $\Gamma$ is horizontal and it follows that $((p'^*N) \cdot \Gamma) > 0$. Therefore, $(\Sigma \cdot \Gamma) < 0$ and $\Gamma \subset \Sigma$. By \eqref{restriction} again, we conclude that $\Gamma$ has to be a ruling.
	\end{proof}
	
	Now we follow the argument in \cite[(3.4)]{Kobayashi} to get the minimal model $\phi: X' \to X$ of $X'$ induced by the linear system $|mM|$ for $m \gg 0$. Here $\phi$ is just a contraction of all rulings on $\Sigma$. Notice that $X$ is smooth, $K_X$ is ample and $\phi^*K_X = M$. Moreover, $\phi$ induces a fibration
	$$
	f: X \to B
	$$
	such that $f' = \phi \circ f$. Notice that by Lemma \ref{lemma1},
	$$
	K_{X/B}^3 = (\phi^*K_{X/B})^3 = (M - f'^*K_B)^3 = 4e.
	$$
	
	\begin{lemma}
		We have $\deg f_* \omega_{X/B} = 3e$.
	\end{lemma}
	
	\begin{proof}
		Recall that $K_{X'} = \phi^*K_X + \Sigma$, where $\Sigma$ is $\phi$-exceptional. It follows that
		$$
		f_* \omega_{X/B} = f'_*\omega_{X'/B}.
		$$
		First, by the projection formula, we have
		$$
		\pi_*\omega_{X'/B} = \CO_Y(p^*(B'+2eC) + H) \oplus \CO_Y(p^*(B'+2eC-5D) - 2H).
		$$
		Then projecting via $p$, we obtain
		$$
		p_* (\pi_* \omega_{X'/B}) = \CO_S(B' + 2eC) \oplus \CO_S(-B').
		$$
		Finally, we deduce that
		$$
		{p_B}_* \left(p_* (\pi_* \omega_{X'/B}) \right) = \CO_B(2eP) \oplus \CO_B(2eP - D_B),
		$$
		where $P = p_B(C)$. Thus the proof is completed since $f' = p_B \circ p \circ \pi$.
	\end{proof}
	
	\subsection{Proof of Proposition \ref{example}}
	Now the proof is obvious, since $K^3_{X/B} = 4e$ and $\deg f_*\omega_{X/B} = 3e$ by the previous subsection. Notice that in this case, the general fiber $F$ of $f$ is a $(1, 2)$-surface.
	
	\begin{remark} \label{remark: (1,2) surface}
		In the above example, if we assume $g(B) = 1$, then we have $K_X^3 = 4e$ and $\deg f_*\omega_X = 3e$. Since $h^1(F, \CO_F) = 0$ and $R^1f_* \omega_X$ is torsion free, we deduce that $R^1f_* \omega_X = 0$. By \cite[Lemma 2.4]{Ohno}, $\deg f_* \omega_X = \chi(X, \omega_X)$. Therefore, we obtain an irregular $3$-fold $X$ of general type with $K_X^3 = \frac{4}{3}\chi(X, \omega_X)$.
	\end{remark}

	\section{Set up and some known cases}
	
	\subsection{Setting} \label{setting}
	From this section till the end of this paper, we always assume that $f: X \to B$ is a relatively minimal $3$-fold fibration over a curve with general fiber $F$ a surface of general type. Since terminal singularities in dimension three are isolated, we may assume that $F$ is actually  (smooth) minimal.
	
	\subsection{Known cases}
	In this subsection, we apply the result of Ohno \cite{Ohno} to prove Theorem \ref{main} when numerical invariants of $F$ are not small.
	
	\begin{prop}\label{prop:general case}
		Let $f: X \to B$ be as above. Then we have
		$$
		K_{X/B}^3\ge 2\mathrm{deg}f_{*}\omega_{X/B}
		$$
		unless $F$ belongs to one of the following three types:
		\begin{itemize}
			\item[(i)] $(K_F^2, p_g(F))=(1,1)$;
			\item[(ii)] $(K_F^2, p_g(F))=(1,2)$;
			\item[(iii)] $|K_F|$ is not composed with a pencil and $p_g(F) = 3$.
		\end{itemize}
	\end{prop}
	
	\begin{proof}
		We first consider the case when $p_g(F)=0$. In this case, $f_{*}{\omega_{X/B}}=0$. Thus the result holds trivially.
		
		When $p_g(F) = 1$ and $K_F^2 > 1$, we have $K_{X/B}^3\ge 2\mathrm{deg}f_{*}\omega_{X/B}$ by \cite[Proposition 2.6]{Ohno}.
		
		In the following, we assume that $p_g(F) \ge 2$. If $|K_F|$ is composed with a pencil and $(K_F^2, p_g(F)) \ne (1, 2)$, by \cite[Proposition 2.1(1)]{Ohno}, we have
		$$
		K_{X/B}^3\ge\frac{4(p_g(F)-1)}{p_g(F)}\mathrm{deg}f_{*}\omega_{X/B}\ge 2\mathrm{deg}f_{*}\omega_{X/B}.
		$$
		If $|K_F|$ is not composed with a pencil (a priori, $p_g(F) \ge 3$) and $p_g(F) \ge 4$, by \cite[Proposition 2.1(1)]{Ohno}, we have
		$$
		K_{X/B}^3\ge\frac{4(p_g(F)-2)}{p_g(F)}\mathrm{deg}f_{*}\omega_{X/B}\ge 2\mathrm{deg}f_{*}\omega_{X/B}.
		$$
		
		As a result, there are three exceptions in total as described in the proposition, and the proof is completed.
	\end{proof}
	
	\subsection{Integral model} \label{interal model}
	
	Let $f: X \to B$ be as in \S \ref{setting}. Let $\pi: Y \to X$ be a resolution of singularities of $X$, and let $\tilde{f}: Y\to B$ be the induced fibration. We will set up the integral model of $f$ in this section, and this integral model will be used extensively in this paper.
	
	By the Lefschetz principle, we may assume that $f$ is defined over a finite extension $k$ over $\QQ$. Let $\mathcal{Z}$ be a scheme of finite type over $\ZZ$ with  function field $k$. Let $\mathcal{Y}\to\mathcal{X} \to \mathcal{B} \to \mathcal{Z}$ be a projective and flat morphism extending $\tilde{f}$. Replacing $\mathcal{Z}$ by a Zariski open set and taking a finite \'etale cover if necessary, we may assume that the generic fiber $\tilde{f}: Y \to B$ has an integral model $\mathcal{Y}\to \mathcal{X} \to \mathcal{B} \to \mathcal{Z}$ such that
	\begin{itemize}
		\item[(i)] the morphisms $\mathcal{Y} \to \mathcal{Z}$ and $\mathcal{B}\rightarrow\mathcal{Z}$ are projective and smooth;
		\item[(ii)] $\tilde{f}: \mathcal{Y} \to \mathcal{B}$ and $f: \mathcal{X}\to \mathcal{B}$ are projective, flat and Cohen-Macaulay morphisms of pure relative dimension two;
		\item[(iii)] for any closed point $z \in \mathcal{Z}$, the general fiber of $f_z: \mathcal{X}_z \to \mathcal{B}_z$ is a $(K_F^2, p_g(F))$ surface.
	\end{itemize}
	
	In the following context, we usually fix an ample and effective $\QQ$-divisor $A$ on $B$. By the base-point-freeness theorem (see \cite[Theorem 3.3]{Kollar_Mori} for instance), the $\QQ$-divisor $K_{X/B} + f^*A$ is semi-ample.\footnote{In fact, we may choose an integer $a \gg 0$ such that $aA$ is integral with $\deg(aA - K_B) > 0$. Then the two divisors $a(K_{X/B} +f^*A) - K_X$ and $K_{X/B} + f^*A$ are both nef and big. Thus the semi-ampleness of $K_{X/B} + f^*A$ follows from the base-point-freeness theorem.} Similarly to \cite[Section 4B]{Patakfalvi}, we may assume that
	\begin{itemize}
		\item [(iv)] the divisor $A$ extends to a relatively ample and effective $\QQ$-divisor $\mathcal{A}$ on $\mathcal{B}$, and $K_{\mathcal{X}_z/\mathcal{B}_z} + f_z^*\mathcal{A}_z$ is semi-ample. In particular, $\pi_z^*(K_{\mathcal{X}_z/\mathcal{B}_z} + f_z^*\mathcal{A}_z)$ is a semi-ample $\QQ$-divisor on $\mathcal{Y}_z$.
	\end{itemize}
	
	The above setting (i)--(iv) will be used throughout the rest of the paper.

	\section{Proof for $(1, 2)$-surface fibers} \label{12proof}
	
	Let $f: X \to B$ and $F$ be as in \S \ref{setting}. In this section, we mainly focus on the case when $F$ is a $(1, 2)$-surface. This is the first exceptional case as shown in Proposition \ref{prop:general case}, and it turns out that this is the most difficult case, at least to us.
	
	\subsection{Resolving base loci of linear systems} We need the following lemma.
	\begin{lemma} {\cite[Lemma 4.2]{Chen}} \label{lem:resolution}
		Suppose that $Z$ is a smooth projective $3$-fold. Let $M$ be a divisor on $Z$ such that $h^0(Z, M) \ge 2$ and that the base locus of $|M|$ has  codimension at least two. Then
		there are successive blow-ups
		$$
		\beta: Y = X_{n+1} \stackrel{\pi_n} \to X_{n} \to \cdots \to X_{i+1} \stackrel{\pi_{i}} \to X_{i} \to \cdots \to X_1 \stackrel{\pi_0} \to Z
		$$
		such that $\pi_i$ is a blow-up along a smooth irreducible center $W_i$, where $W_i$ is contained in the base locus of the movable part of $(\pi_0\circ \pi_1 \circ \cdots \circ \pi_{i-1})^*|M|$.
		Moreover, the morphism $\beta=\pi_0 \circ \cdots \circ \pi_n$ satisfies the following properties:
		\begin{enumerate}
			\item Denote by $N$ the movable part of $\beta^*|M|$. Then $|N|$ is base point free.
			\item The following formulae
			\begin{equation} \label{eq:blowupadjunction}
			K_{Y}=\beta^*K_Z+\sum_{i=0}^{n}a_i E_i,\quad  \beta^*M = N+\sum_{i=0}^{n}b_i E_i
			\end{equation}
			hold. Here for any $0 \le i \le n$, $E_i$ is the strict transform of the exceptional divisor of $\pi_i$, and
			$a_i$ and $b_i$ are positive integers satisfying $a_i \le 2b_i$.
		\end{enumerate}
		
	\end{lemma}
	
	\begin{remark} \label{remarkcharp}
		Lemma \ref{lem:resolution} holds also when the base field has positive characteristics, and this is crucial to us.
	\end{remark}
	
	\subsection{The case when $g(B) \le 1$}
	We first consider the case when $g(B)=0$.
	\begin{prop}\label{prop:(1,2) surf b0}
		Let $f: X \to B$ and $F$ be as in \S \ref{setting}. Suppose that $X$ is smooth, $F$ is a $(1, 2)$-surface, and $g(B) = 0$. Then
		$$
		K_{X/B}^3 \ge \frac{4}{3} \deg f_{*}\omega_{X/B}.
		$$
	\end{prop}
	
	\begin{proof}
		Since $f_{*}\omega_{X/B}$ is a semi-positive vector bundle of rank $2$ and $B \cong \PP^1$, we may assume that
		$$
		f_{*}\omega_{X/B}=\CO_B (a)\oplus \CO_B (b),
		$$
		where $a$ and $b$ are two non-negative integers. We may further assume that $\deg f_* \omega_{X/B} = a+b > 0$. Thus $h^0(X, K_{X/B})= a+b+2 \ge 3$.
		
		Since $h^1(F, \CO_F)=0$, we have $R^1 f_{*}\omega_{X}=0$. Together with Koll\'ar's vanishing theorem, we know that
		$$
		h^1(X, K_X + F) = h^1(B, f_*\omega_X \otimes \CO_B(P))= 0,
		$$
		where $P=f(F)$ is a general point on $B$. Thus we have a surjective map
		$$
		H^0(X, K_{X/B}) \twoheadrightarrow H^0(F, K_F).
		$$
		Since $h^0(X, K_{X/B}) \ge 3$, we know from the above surjectivity that $h^0(K_X+F) \ge 1$. This implies that $|K_{X/B}|$ separates two distinct fibers of $f$. Thus the map given by the complete linear system $|K_{X/B}|$ has the image of dimension two.
		
		Write
		$$
		|K_{X/B}|=|M|+Z,
		$$
		where $|M|$ and $Z$ are respectively the movable and fixed part of $|K_{X/B}|$.
		Since $|K_{X/B}||_F=|K_F|$ and $|K_F|$ has only one isolated base point, we see that the base locus of $|M|$ contains a section of $f$. It also implies that $Z$ is vertical with respect to $f$. Take a birational modification $\beta: Y \rightarrow X$ as in Lemma \ref{lem:resolution}, and write
		$$
		K_Y = \beta^*K_X + E, \quad \beta^*M = N + E'.
		$$
		Here $|N|$ is base point free, and $E$ and $E'$ are both linear combinations of exceptional divisors as in Lemma \ref{lem:resolution} so that $E\le 2E'$.
		
		Take a general member $S \in |N|$. By Bertini's theorem, we may assume that $S$ is a smooth surface. Write $f'= f \circ \beta$ and denote by $F'$ the general fiber of $f'$. Let $\sigma: F' \rightarrow F$ be the natural contraction. Notice that the morphism induced by $|N|$, when restricting on $S$, coincides with the restricted fibration $f'|_S$ over $B$. Thus we can write $N|_S \equiv tC$, where $C$ is a general fiber of $f'|_S$ and $t \ge h^0(Y, N) - 2 = h^0(X, K_{X/B}) -2$. Since $N|_{F'}$ is just the movable part of $|K_{F'}|$ (which is the same as the movable part of $|K_F|$), we deduce that $C$ is actually a smooth curve of genus $2$. By the adjunction formula, we have $(K_Y|_S \cdot C) = ((K_Y + S)|_S \cdot C) = (K_S\cdot C) = 2$ and
		\begin{equation} \label{(1,2)b=0a}
		((\beta^*K_X)|_S \cdot C) = ((\beta^*K_{X/B})|_S\cdot C) =((\sigma^*K_F) \cdot S|_{F'}) = 1.
		\end{equation}
		The above two equalities imply that
		$$
		(E|_S \cdot C) = (E'|_S \cdot C) = 1.
		$$
		Using the fact that $E|_S\le 2E'|_S$ by Lemma \ref{lem:resolution}, we conclude that the horizontal part of $E|_S$ and that of $E'|_S$ have to be the same, which we will denote by $\Gamma$. In particular, $\Gamma$ is a section of $f'|_S$, i.e., $g(\Gamma) = 0$. Write
		$$
		E|_S=\Gamma+E_V, \quad E'|_S=\Gamma+E'_V,
		$$
		where $E_V$ and $E'_V$ are vertical divisors on $S$ with respect to $f'|_S$. Then $E_V \le 2E'_V$. Thus it follows by the adjunction formula that
		\begin{eqnarray*}
			-2 & = & ((K_Y|_S+S|_S+\Gamma)\cdot \Gamma) \\
			& = & (((\beta^*K_X)|_S+S|_S + 2 \Gamma + E_V) \cdot \Gamma) \\
			& \le & (((\beta^*K_X)|_S + S|_S + 2\Gamma + 2E'_V + 2(\beta^*Z)|_S) \cdot \Gamma) \\
			& \le & ((2(\beta^*K_{X/B})|_S + (\beta^*K_X)|_S - S|_S) \cdot \Gamma) \\
			& = & 3((\beta^*K_{X/B})|_S \cdot \Gamma) - t - 2.
		\end{eqnarray*}
		That is,
		\begin{equation} \label{(1,2)b=0b}
		((\beta^*K_{X/B})|_S \cdot \Gamma) \ge \frac{1}{3} t.
		\end{equation}
		As a result, we deduce from \eqref{(1,2)b=0a} and \eqref{(1,2)b=0b} that
		\begin{eqnarray*}
			K_{X/B}^3 & \ge & ((\beta^*K_{X/B})|_S)^2 \\
			& \ge & ((\beta^*K_{X/B})|_S \cdot S|_S) + ((\beta^*K_{X/B})|_S \cdot E'|_S) \\
			& \ge & \frac{4}{3}t.
		\end{eqnarray*}
		Recall that we have known that $t \ge h^0(X, K_{X/B}) - 2 = a + b = \deg f_* \omega_{X/B}$. Thus the proof is completed.
	\end{proof}
	
	\begin{remark}
		Some ideas in the above proof are inspired by the work of Chen. In \cite[Theorem 4.7]{Chen}, he used them to study the canonical linear system of $3$-folds and deduced a Noether inequality. We realized that they also apply to the current setting.
	\end{remark}
	
	\begin{remark}
		In fact, Proposition \ref{prop:(1,2) surf b0} holds when $X$ has Gorenstein singularities, and the corresponding proof is exactly the same, just replacing the use of Lemma \ref{lem:resolution} by \cite[Theorem 2.5]{Chen_Chen} for the Gorenstein case.
	\end{remark}
	
	Our next goal is to study the case when $g(B) = 1$.
	
	\begin{prop}\label{prop:(1,2) surf b1}
		Let $f: X \to B$ and $F$ be as in \S \ref{setting}. Suppose that $X$ is Gorenstein, $F$ is a $(1, 2)$-surface, and $g(B) = 1$. Then
		$$
		K_{X/B}^3 \ge \frac{4}{3} \deg f_{*}\omega_{X/B}.
		$$
	\end{prop}
	
	\begin{proof}
		In this case, $K_{X/B}$ and $K_X$ coincide. To prove this result, we may assume that $\deg f_*\omega_{X/B} = \deg f_*\omega_X > 0$. This implies that $X$ is a minimal $3$-fold of general type (see \cite[Proposition 4.6]{Viehweg} for instance). Thus by \cite[Theorem 1.5]{Zhang}, we have
		$$
		K_X^3 \ge \frac{4}{3} \chi(X, \omega_{X}).
		$$
		
		On the other hand, since $R^1 f_*\omega_X = 0$, from the Leray spectral sequence, we deduce that
		\begin{eqnarray*}
			h^1(X, \omega_X) & = & h^1(B, f_* \omega_X); \\
			h^2(X, \omega_X) & = & h^0(B, R^2 f_* \omega_X) = h^0(B, \omega_B) = 1.
		\end{eqnarray*}
		By the Riemann-Roch theorem, the above two equalities imply that
		$$
		\deg f_*\omega_X = h^0(B, f_*\omega_X) - h^1(B, f_*\omega_X) = \chi(X, \omega_X).
		$$
		Thus the proof is completed.
	\end{proof}
	
	\subsection{The case when $g(B) \ge 2$}
	The main result in this subsection is the following.
	\begin{prop}\label{prop:(1,2) surf b2}
		Let $f: X \to B$ and $F$ be as in \S \ref{setting}. Suppose that $X$ is smooth, $F$ is a $(1, 2)$-surface, and $b=g(B) \ge 2$. Then
		$$
		K_{X/B}^3 \ge \frac{4}{3} \deg f_{*}\omega_{X/B}.
		$$
	\end{prop}
	Our proof here utilizes the mod $p$ reduction argument in a very essential way. The rest of this section will be devoted to the proof of Proposition \ref{prop:(1,2) surf b2}.
	
	\subsubsection{Some easy cases} Here we prove Proposition \ref{prop:(1,2) surf b2} for some cases when the characteristic zero argument applies.
	
	Similarly to the proof of Proposition \ref{prop:(1,2) surf b1}, we know that $X$ is minimal of general type, and we have
	$$
	h^1(X, \omega_X) = h^1(B, f_*\omega_X), \quad h^2(X, \omega_X) = b.
	$$
	Since $f_* \omega_{X/B}$ is of rank $2$, we have
	\begin{eqnarray*}
		\deg f_* \omega_{X/B} & = & \deg f_* \omega_X - 4(b-1) \\
		& = & h^0(B, f_*\omega_X) - h^1(B, f_*\omega_X) - 2(b-1) \\
		& = & \chi(X, \omega_X) - 3(b-1).
	\end{eqnarray*}
	Also recall that in this case,
	$$
	K_{X/B}^3 = K_X^3 - 6(b-1).
	$$
	
	\begin{lemma}
		Proposition \ref{prop:(1,2) surf b2} holds when $h^0(X, K_X) \le 2$.
	\end{lemma}
	
	\begin{proof}
		If $h^0(X, K_X) = 0$, we have $K_X^3\ge 2\chi(X, \omega_X)$ by \cite[3.2 Case 2, p761]{Chen_Hacon}. Thus we deduce that
		$$
		K_{X/B}^3 \ge 2 \deg f_{*}\omega_{X/B}
		$$
		from the above calculation.
		
		If $1 \le h^0(X, K_X) \le 2$, then $\chi(X, \omega_X)\le \frac{3}{2}h^0(X, K_X)$ by \cite[Proposition 2.1]{Chen_Hacon}. In particular, $\chi(X, \omega_X)\le 3$. On the other hand, $K_{X/B}^3\ge 0$ implies that $K_X^3 \ge 6$. Therefore, $K_X^3\ge2\chi(\omega_X)$. Thus the result follows similarly to the above.
	\end{proof}

	The general case that we cannot handle so far using characteristic zero arguments is when $h^0(X, K_X) \ge 3$. However, we discover that a characteristic $p>0$ method is applicable here quite well.
	
	\subsubsection{Positive characteristic preparations} \label{charppreparation}
	Let $k$ be a field of characteristic $p > 0$. We have the following result which is a positive characteristic counterpart of a result of Chen \cite[Lemma 4.5]{Chen}. One important difference from \cite[Lemma 4.5]{Chen} is that we have to consider Gorenstein $3$-folds rather than merely smooth ones, because in the future, we need to apply this to the Frobenius base change of the fibered $3$-folds. These $3$-folds are by no means smooth.
	
	\begin{lemma}\label{lem:(1,2) surface not pencil}
		Let $f:\Sigma\to B$ be a fibration defined over an algebraically closed field from a Gorenstein integral $3$-fold $\Sigma$ to a nonsingular curve $B$ of genus $b \ge 2$ with general fiber $F$ a $(1,2)$-surface. Suppose that $f_{*}\omega_{\Sigma/B}$ is semi-positive and that $h^0(\Sigma, K_{\Sigma})\ge 3$. Then $|K_{\Sigma}|$ is not composed with a pencil.
	\end{lemma}

	\begin{proof}
		We prove by showing a contradiction. Suppose $|K_{\Sigma}|$ is composed with a pencil. Take a birational modification $\pi:\tilde{\Sigma}\to \Sigma$ such that:
		\begin{enumerate}
			\item[(1)] $\tilde{\Sigma}$ is nonsingular;
			\item[(2)] the movable part $|M|$ of $\pi^*|K_{\Sigma}|$ is base point free. Denote by $g$ the morphism induced by $|M|$ and by $S$ a general member in $|M|$.
		\end{enumerate} 
		
		By our assumption, $g$ is a morphism onto a curve. Let $g=s\circ h$ be the Stein factorization of $g$, where $h:\tilde{\Sigma}\to C$ is a fibration onto a nonsingular curve. Denote by $F'$ a general fiber of $h$. Then we have 
		$$
		S\equiv aF',
		$$ 
		where $a\ge h^0(\Sigma, K_{\Sigma})-1 \ge 2$.
		
		Let $\tilde{f}=f\circ \pi$. Denote by $\tilde{F}$ a general fiber of $\tilde{f}$ and by $\sigma: \tilde{F}\to F$ the natural contraction morphism.
		
		We first prove that $h$ and $\tilde{f}$ give the same fibration. Otherwise, $h|_{\tilde{F}}$ is a fibration onto $C$. Thus we would have $((\sigma^*K_F) \cdot F'|_{\tilde{F}}) \ge 1$, so that
		$$
		1 = (\sigma^*K_F)^2 = ((\sigma^* K_F )\cdot (\pi^*K_{\Sigma})|_{\tilde{F}}) \ge a ((\sigma^*K_F) \cdot F'|_{\tilde{F}}) \ge a \ge 2.
		$$
		This is a contradiction. 
		
		Denote by $\mathcal{L}_0$ the saturated subbundle of $f_{*}\omega_{\Sigma}$ generated by global sections. Since $h$ and $\tilde{f}$ give the same fibration, $\mathcal{L}_0$ is a line bundle. Then we have the following exact sequence:
		$$
		0\to \mathcal{L}_0\to f_{*}\omega_{\Sigma}\to \mathcal{L}_1\to 0,
		$$
		where $\mathcal{L}_1$ is a line bundle since $f_* \omega_\Sigma$ is of rank two. In particular, we have $h^0(B, \mathcal{L}_0)=h^0(\Sigma, K_{\Sigma})\ge 3$. This also implies that $h^0(B, \mathcal{L}_1) \le h^1(B, \mathcal{L}_0)$.
		
		When $h^1(B, \mathcal{L}_0)=h^0(B, \omega_B\otimes\mathcal{L}_0^{\vee})>0$, we have
		\begin{equation} \label{eq:pencil1}
		h^0(B, \omega_B\otimes\mathcal{L}_0^{\vee}))\le b+1-h^0(B, \mathcal{L}_0)\le b-2.
		\end{equation}
		Here the first inequality follows from the fact that 
		$$
		h^0(B, \omega_B\otimes\mathcal{L}_0^{\vee})+h^0(B, \mathcal{L}_0) - 1 \le h^0(B, \omega_B).
		$$ 
		Thus $h^0(B, \mathcal{L}_1) \le h^1(B, \mathcal{L}_0) \le b-2$. By the Riemann-Roch theorem on $B$, we have 
		$$
		\deg \mathcal{L}_1 \le h^0(B, \mathcal{L}_1)+b-1\le 2b - 3.
		$$ 
		As a result, we deduce that $\deg(\mathcal{L}_1\otimes \omega_B^{\vee})\le -1$. This gives rise to a contradiction as needed, because $f_*\omega_{\Sigma/B}$ is semi-positive and each of its quotient bundles should have  non-negative degree.
		
		The proof is now completed.
	\end{proof}

	In the following, let $f: \Sigma \to B$ be a fibration defined over $k$ from a smooth $3$-fold $\Sigma$ to a curve $B$ of genus $b$. Suppose that $\Sigma$ is of general type, and the general fiber $F$ of $f$ is a $(1, 2)$-surface over $k$. We further assume that $F$ is ``specialized from characteristic zero", i.e., it is a general closed fiber of an integral model of a $(1, 2)$-surface in characteristic zero, so that its geometry is the same as that in characteristic zero.
	
	Let $L$ be nef and big $\QQ$-divisor on $\Sigma$ such that
	\begin{itemize}
		\item $L-K_\Sigma$ is a vertical $\QQ$-divisor with respect to $f$;
		\item $\rounddown{L}|_F \sim K_F$;
		\item the map induced by the complete linear system $|\rounddown{L}|$ has the image of dimension two.
	\end{itemize}
	The above assumption guarantees that the map $H^0(\Sigma, \rounddown{L}) \to H^0(F, K_F)$ is surjective.
	
	Write
	$$
	L=M+Z,
	$$
	where $|M|$ is the movable part of $|\rounddown{L}|$ and $Z=L-M$ is an effective $\QQ$-divisor. Similarly to the proof of Proposition \ref{prop:(1,2) surf b0}, the base locus of $|M|$ contains a unique section of $f$. Since $L-K_\Sigma$ is  vertical and $M|_F \sim \rounddown{L}|_F \sim K_F$, we know that $Z$ is vertical with respect to $f$. We denote this distinguished section of $f$ by $\Gamma_0$. For any fiber $F$ of $f$, since $\Gamma_0$ is the section of $f$, $\Gamma_0 \cap F$ consists of exactly one single point. 
	
	\textbf{Claim I}: For any fiber $F$ of $f$, $\Gamma_0 \cap F$ is a smooth point on $F$.
	
	Denote by $P \in F$ this point and by $\mu \colon \Sigma'\rightarrow \Sigma$ the blow-up of $\Sigma$ at $P$. Since $\Sigma$ is smooth, $\Sigma'$ is also smooth. Write
	$$
	\mu^*F = F' + r E,
	$$
	where $E$ is the exceptional divisor of $\mu$, $F'$ is the strict transform of $F$, and $r$ is the multiplicity of $F$ at $P$. Denote by $\Gamma'_0$ the strict transform of $\Gamma_0$. Then $\Gamma'_0$ is isomorphic to $\Gamma_0$. Moreover, $\Gamma'_0$ is a section of the fibration $f \circ \mu$. We have
	$$
	1= ((\mu^*F) \cdot \Gamma'_0) = (F'\cdot \Gamma'_0) + r (E\cdot \Gamma'_0).
	$$
	Since $(E\cdot \Gamma'_0)=1$ and $(F'\cdot \Gamma'_0) \ge 0$, we have $r=1$. Because $\Sigma$ is smooth, the point $P$ must be a \emph{smooth point} on $F$. Thus we finish the proof of the claim. 
	
	By Remark \ref{remarkcharp}, we can take a birational modification $\beta: \Sigma_0 \rightarrow \Sigma$ as in Lemma \ref{lem:resolution}, and write
	$$
	K_{\Sigma_0} = \beta^*K_\Sigma + E, \quad \beta^*M = N + E'
	$$
	so that $|N|$ is base point free and that both $E$ and $E'$ are linear combinations of strict transforms of exceptional divisors $E_i$ and $E\le 2E'$.
	
	Let $f'= f \circ \beta$ and $F'$ be a general fiber of $f'$. Denote by $\sigma: F' \to F$ the natural contraction. Take a general member $S \in |N|$. By Bertini's theorem (see \cite[Thm. I.6.3]{Zariski}), $S$ is irreducible. Since $|N||_{F'}$ is the movable part of $|\sigma^*K_F|$ and $F$ is a (1,2) surface, $S$ is reduced by \cite[Thm. I.6.3]{Zariski}. That is, $S$ is a Gorenstein integral surface, but not necessarily smooth. The restriction $f'|_S: S \to B$ is again a fibration with general fiber $C$ a smooth curve of genus $2$. We also have $N|_S \equiv tC$, where $t \ge h^0(\Sigma_0, N) - 2 = h^0(\Sigma, \rounddown{L}) -2$. Similar as in the proof of Proposition \ref{prop:(1,2) surf b0}, we deduce that
	\begin{itemize}
		\item $((\beta^*K_\Sigma)|_S \cdot C) = ((\beta^*L)|_S \cdot C) =((\sigma^*K_F) \cdot S|_{F'}) = 1$;
		\item $(E|_S \cdot C) = (E'|_S \cdot C) = 1$.
	\end{itemize}
	Thus the horizontal part of $E|_S$ is the same as that of $E'|_S$, which we denote again by $\Gamma$. Here $\Gamma$ is a section of $f'$, and $g(\Gamma) = b \ge 2$. We emphasize that
	$$
	\beta(\Gamma) = \Gamma_0.
	$$
	Therefore, from Claim I, we know that for any fiber $F'$ of $f'$, the point $\Gamma \cap F'$ is a smooth point on $F'$. 
	
	\textbf{Claim II}: $\Gamma$ is a Cartier divisor on $S$. 
	
	We may assume that $\Gamma$ comes from some $E_i$, i.e., $\Gamma \subseteq E_i|_S$. Since $(E_i \cdot C)=1$ and $|N||_{E_i}$ separates fibers of $f'|_{E_i}$, we know that $|N||_{E_i}$ induces a generic finite morphism on $E_i$. By Bertini's theorem \cite[Thm. I.6.3]{Zariski}, we may further take the above $S$ general so that $S\cap E_i$ is irreducible. Then $E_i|_S$ is an irreducible curve on $S$. Since $(E_i|_S\cdot C)=1$, $E_i|_S$ is also reduced. We deduce that $\Gamma=E_i|_S$. In particular, $\Gamma$ is an effective Cartier divisor on $S$.
	
	Let $E_V$ and $E'_V$ be the same as before so that $E_V \le 2E'_V$.
	
	\begin{lemma} \label{lemmacharp}
		Let $l = ((L-K_{\Sigma/B}) \cdot \Gamma_0)$. Then we have
		$$
		L^3 \ge \frac{4}{3} \left( h^0(\Sigma, \rounddown{L}) - 2 \right) + \frac{1}{3}l.
		$$
	\end{lemma}
	
	\begin{proof}
		Since $\Gamma$ is a Cartier divisor on a Gorenstein integral surface $S$, by the adjunction formula, we have
		\begin{eqnarray*}
			2b-2 & = & ((K_{\Sigma_0}|_S+S|_S+\Gamma)\cdot \Gamma) \\
			& = & (((\beta^*K_\Sigma)|_S+S|_S + 2 \Gamma + E_V) \cdot \Gamma) \\
			& \le & (((\beta^*K_\Sigma)|_S + S|_S + 2\Gamma + 2E'_V + 2(\beta^*Z)|_S) \cdot \Gamma) \\
			& = & ((2(\beta^*L)|_S + (\beta^*K_\Sigma)|_S - S|_S) \cdot \Gamma) \\
			& = & 3((\beta^*L)|_S \cdot \Gamma) - t - ((\beta^*L - \beta^*K_\Sigma) \cdot \Gamma) \\
			& = & 3((\beta^*L)|_S \cdot \Gamma) - l - t + 2b - 2,
		\end{eqnarray*}
		i.e.,
		$$
		((\beta^*L)|_S \cdot \Gamma) \ge \frac{1}{3} (t + l).
		$$
		Therefore, we deduce that
		\begin{eqnarray*}
			L^3 & \ge & ((\beta^*L)|_S)^2 \ge ((\beta^*L)|_S \cdot S|_S) + ((\beta^*L)|_S \cdot E'|_S) \\
			& \ge & \frac{4}{3} t + \frac{1}{3} l = \frac{4}{3} \left( h^0(\Sigma, \rounddown{L}) - 2 \right) + \frac{1}{3}l.
		\end{eqnarray*}
		Thus the proof is completed.
	\end{proof}
	
	\subsubsection{Completing the proof} \label{modpproof} 
	
	Let $f: X \to B$ be as in Proposition \ref{prop:(1,2) surf b2}. Furthermore, we may assume that the canonical linear system $|K_X|$ induces a map to a surface image and that $\deg f_*\omega_{X/B} > 0$. By Koll\'ar's decomposition \cite{Kollar}, we have
	$$
	f_*\omega_{X/B}=Q \oplus P,
	$$
	where $Q$ is an ample vector bundle on $B$ and $P$ is a flat bundle on $B$. The assumption that $\deg f_* \omega_{X/B} >0$ guarantees that $Q \neq 0$. Since $f_* \omega_{X/B}$ is of rank $2$, either $P$ is a line bundle or $P = 0$.

	Fix an ample and effective $\QQ$-divisor $A$ on $B$. We extend $f: X \to B$ to $\CX \to \CB \to \CZ$ as in \S \ref{interal model}. The divisor $A$ also extends to a relatively ample and effective $\QQ$-divisor $\CA$ on $\mathcal{B}$. We may assume that  (i)-(iv) in \S \ref{interal model} hold.  Similarly to \cite[Section 4B]{Patakfalvi}, by shrinking $\CZ$ if necessary, we may further assume that
	\begin{itemize}
		\item[(v)] the ample bundle $Q$ extends to a relatively ample bundle $\mathcal{Q}$ on $\CB$ and the flat bundle $P$ extends to a bundle $\mathcal{P}$ on $\CB$. 
	\end{itemize}
	Moreover, we still have 
	$$
	f_* \omega_{\mathcal{X}/\mathcal{B}} =\mathcal{Q}\oplus\mathcal{P}.
	$$ 
	In particular, for any closed point $z \in \mathcal{Z}$, we have 
	$$
	{f_z}_*\omega_{X_z/B_z}=\mathcal{Q}_z\oplus\mathcal{P}_z,
	$$ 
	where $\mathcal{Q}_z$ is ample, and either $\mathcal{P}_z=0$ or $\mathcal{P}_z$ is a line bundle of degree $0$. Nevertheless, ${f_z}_*\omega_{X_z/B_z}$ is semi-positive.
	
	Now take a closed point $z \in \mathcal{Z}$ whose residue field has characteristic $p > 0$. For simplicity, we denote
	$$
	X_0 := \mathcal{X}_z, \quad B_0 := \mathcal{B}_z, \quad A_0 := \mathcal{A}_z
	$$
	in the following. We still use $f$ to denote the fibration $X_0 \to B_0$ and use $F$ to denote a general fiber of $f$. Notice that the fibration $f: X_0 \to B_0$ is not necessarily relatively minimal, because it is not clear to us whether $K_{X_0/B_0}$ is still nef.\footnote{There are some known examples in which none of the reductions of a nef divisor to positive characteristics is nef. See \cite{Langer}.} 
	
	In the following, let
	$$
	M = K_{X_0/B_0} + f^*A_0.
	$$
	Let $F^e$ be the $e^{\mathrm{th}}$ absolute Frobenius morphism of $B_0$. Consider the following commutative diagram:
	$$
	\xymatrix{
		X''_{e} \ar[r]^{\varepsilon''} \ar[rrd]_{f''_e} \ar@/^2.5pc/[rrr]^{\varepsilon} & X'_{e} \ar[dr]^{f'_e} \ar[r]^{\varepsilon'} & X_e \ar[d]^{f_e} \ar[r]^{F^e_f} & X_0 \ar[d]^f \\
		& & B_0 \ar[r]^{F^e} & B_0      }
	$$
	Here $X_e = X_0 \times_{F^e} B_0$, $\varepsilon'$ is the normalization of $X_e$, $\varepsilon''$ is a resolution of singularities of $X'_{e}$ which is an isomorphism over the smooth locus of $X'_e$, and $\varepsilon$ is the composition of $\varepsilon''$, $\varepsilon'$ and $F^e_f$. Notice that the existence of $\varepsilon''$ as required is known by \cite[Theorem 1.1]{Cutkosky} for instance. The general fibers of $f_e$ and $f''_e$ are both isomorphic to $F$. By abuse of notation, we still denote them by $F$. Let 
	$$
	M_e = F_f^{e*}M, \quad M''_e = \varepsilon^*M.
	$$ 
	Then it is easy to see that $M''_e$ is nef and big. Let 
	$$
	L_e=M_e + 2(b-1)F, \quad L''_e = M''_e  + 2(b-1)F.
	$$ 
	
	Recall that we have assumed that $\deg f_* \omega_{X_0/B_0} = \deg f_* \omega_{X/B} > 0$. When $e$ is sufficiently large, we have 
	$$
	h^0(X_e, K_{X_e}) \ge h^0(B, {f_e}_*\omega_{X_e/B_0}) \ge \deg {f_e}_*\omega_{X_e/B_0} - 2(b-1) \ge 3.
	$$
	Since ${f_e}_* \omega_{X_e/B}= {F^e}^*f_*\omega_{X/B}$ is still semi-positive, by Lemma \ref{lem:(1,2) surface not pencil}, we know that $|K_{X_e}|$ is not composed with a pencil.\footnote{Notice that here $X_e$ is Gorenstein but not smooth in general.} Notice that $\rounddown{L_e}\ge K_{X_e}$. Thus we have $|\rounddown{L''_e}||_F = |K_F|$ and the map induced by the complete linear system $|\rounddown{L''_e}|$ has the image of dimension two. Thus we are under the setting of \S \ref{charppreparation}.
	
	Let $|N|$ be the movable part of $|K_{X_0}|$. As is described in \S \ref{charppreparation}, there is a unique section $\Gamma_0$ of $f$ lying in its base locus. Similarly, the movable part $|N_e|$ of $|\rounddown{L''_e}|$ also contains a unique section $\Gamma_e$ of $f''_e$. Moreover, $\varepsilon(\Gamma_e) = \Gamma_0$, because when restricting on general fibers, they are  identical.
	
	By our assumption, it is easy to see that $L_e - K_{X_e}$ is vertical. Notice that the modification of either $\varepsilon'$ or $\varepsilon''$ \emph{only occurs on the singularities of singular fibers of $f_e$ or $f'_e$}. This particularly implies that $K_{X''_e} - \varepsilon''^*(\varepsilon'^*K_{X_e})$ is vertical. Therefore, we deduce $L''_e - K_{X''_e}$ is vertical.
	
	Recall that in \S \ref{charppreparation}, we have showed that $\Gamma_0$ meets every fiber of $f$ at its smooth point. Thus $F_f^{e*}\Gamma_0$ meets each fiber of $f_e$ at its smooth point. Since neither $\varepsilon'$ nor $\varepsilon''$ touches the smooth locus of any fiber, we deduce that $\Gamma_e = \varepsilon^*\Gamma_0$. In particular,
	$$
	((K_{X''_e/B_0} - \varepsilon^*K_{X_0/B_0}) \cdot \Gamma_e) = 0.
	$$
	Thus it follows that
	\begin{eqnarray*}
		((L''_e - K_{X''_e/B_0}) \cdot \Gamma_e) & = & ((L''_e - \varepsilon^*K_{X_0/B_0}) \cdot \Gamma_e) \\
		& = & ((M''_e-\varepsilon^*K_{X_0/B_0}) \cdot \Gamma_e) + 2(b-1)\\
		& = & (p^e (M-K_{X_0/B_0}) \cdot \Gamma_0) + 2(b-1)\\
		& = & p^e \deg A_0+2(b-1).
	\end{eqnarray*}
	
	Now we apply Lemma \ref{lemmacharp} to $L''_e$. Together with the above calculation, it follows that
	\begin{eqnarray*}
		L''^3_e & \ge & \frac{4}{3} \left( h^0(X''_e, \rounddown{L''_e}) - 2 \right) + \frac{p^e}{3} \deg A_0+\frac{2}{3}(b-1) \\
		& \ge & \frac{4}{3} \left(h^0(X_e, K_{X_e/B_0}) - 2\right) + \frac{p^e}{3} \deg A_0+ \frac{2}{3}(b-1) \\
		& \ge& \frac{4}{3} \left( h^0(X_e, f_{e*} \omega_{X_e/B_0}) - 2 \right) + \frac{p^e}{3} \deg A_0  \\
		& \ge & \frac{4}{3} \left( \deg f_{e*} \omega_{X_e/B_0} - 2b \right)  + \frac{p^e}{3} \deg A_0.
	\end{eqnarray*}
	On the other hand, we have
	$$
	L''^3_e  =  \left( M''_e + 2(b-1)F \right)^3 = M''^3_e + 6(b-1) = p^e M^3 + 6(b-1)
	$$
	as well as
	$$
	\deg f_{e*} \omega_{X_e/B_0} = p^e \deg f_* \omega_{X_0/B_0}.
	$$
	Combine all the above together and let $e \to \infty$. It follows that
	$$
	(K_{X_0/B_0} + f^*A_0)^3 = M^3 \ge \frac{4}{3} \deg f_* \omega_{X_0/B_0} + \frac{1}{3} \deg A_0.
	$$
	Go back to characteristic zero. The above inequality just says that
	$$
	(K_{X/B} + f^*A)^3 \ge \frac{4}{3} \deg f_* \omega_{X/B} + \frac{1}{3} \deg A
	$$
	for the original fibration $f: X \to B$. Since $\deg A$ is arbitrary, we deduce that
	$$
	K_{X/B}^3 \ge \frac{4}{3} \deg f_* \omega_{X/B}
	$$
	by letting $\deg A \to 0^+$. Thus the proof is completed.
	
	\section{Proof for $(1, 1)$-surface fibers} \label{11proof}
	
	The main result in this section is the following:
	
	\begin{prop}\label{prop:(1,1) surface}
		Let $f\colon X\rightarrow B$ be as in \S \ref{setting}. Suppose that the general fiber $F$ of $f$ is a $(1,1)$-surface. Then we have
		$$
		K_{X/B}^3\ge 2\mathrm{deg}f_{*}\omega_{X/B}.
		$$
	\end{prop}
	
	The whole section is devoted to the proof of it, and our method is also via a characteristic $p > 0$ argument.
	
	\subsection{Positive characteristic results.} Assume that $k$ is a field of characteristic $p>0$. Let $f\colon \Sigma \rightarrow B$ be a fibration defined over $k$ from a smooth $3$-fold $\Sigma$ to a smooth curve $B$ with general fiber $F$ a minimal surface of general type. Suppose that $F$ is a $(1, 1)$-surface. Similarly to the $(1, 2)$-surface case in \S \ref{charppreparation}, here we also assume that $F$ is ``specialized from characteristic zero", i.e., it is a general closed fiber of an integral model of a $(1, 1)$-surface in characteristic zero. In particular, we may assume that $h^1(F, \CO_F) = 0$. Thus $\chi(F, \CO_F) = 2$ and $h^0(F, 2K_F) = 3$.
	
	\subsubsection{Some facts} We need the following geometric results on $F$:
	
	\textbf{Fact $1$}: \textit{$|2K_F|$ is base point free.} 
	
	In characteristic zero, this was proved by Catanese in \cite[Theorem 1.4]{Catanese}. Since here we assume that $F$ comes from an integral model of a $(1, 1)$-surface in characteristic zero, we are safe to use it under the current setting.
	
	\textbf{Fact $2$}: \textit{Let $D$ be a prime divisor on $F$ such that $h^0(F, D) \ge 2$. Then $(K_F\cdot D) \ge 2$.}
	
	In characteristic zero, this was proved by Chen in \cite[Lemma 2.5]{Chen}. In the following, we present a proof which is valid regardless of the field characteristic.
	
	\begin{proof} [Proof of Fact $2$]
		We argue by contradiction. Suppose that $(K_F\cdot D) \le 1$. Since $p_a(D) \ge 2$, by the parity and the Hodge index theorem, the only possibility that we have is $(K_F\cdot D) = D^2 = 1$. Thus $D \equiv K_F$. Notice that $h^0(F, D) \ge 2$ implies that $D \not\sim K_F$. In particular, $h^0(F, D-K_F) = 0$. Therefore, by the Riemann-Roch theorem,
		$$
		h^0(2K_F - D) \ge \chi(F, \CO_F) + \frac{((2K_F-D)\cdot (K_F-D))}{2} = 2.
		$$
		
		On the other hand, by Fact $1$, $|2K_F|$ induces a generically finite morphism. This particularly implies that the following restriction map
		$$
		H^0(F, 2K_F) \to H^0(D, 2K_F|_D)
		$$
		has  image of dimension at least $2$, because $D$ is movable. Thus 
		$$
		h^0(F, 2K_F - D) \le h^0(F, 2K_F) - 2 = 1.
		$$
		This is a contradiction.
	\end{proof}
	
	\subsubsection{Key lemma} The following lemma is the key to Proposition \ref{prop:(1,1) surface}.
	
	\begin{lemma} \label{lem:(1,1) surface1}
		Let $L$ be a nef $\QQ$-divisor on $\Sigma$ such that $L|_F \sim 2K_F$. Then
		$$
		L^3\ge 4h^0(\Sigma, \rounddown{L})-12.
		$$
	\end{lemma}
	
	\begin{proof}
		To prove this result, we may assume that $h^0(\Sigma, \rounddown{L})\ge 4$. From the following exact sequence
		\begin{align*}
		0\to H^0(\Sigma, \rounddown{L}-F)\to H^0(\Sigma, \rounddown{L})\to H^0(F, 2K_F),
		\end{align*}
		we know that $h^0(\Sigma, \rounddown{L}-F)>0$. Choose a blow-up $\pi: \Sigma'\to\Sigma$ such that the movable part $|M|$ of $\pi^*|\rounddown{L}|$ is base point free. Denote by $F'=\pi^*F$ and by $\sigma:F'\to F$ the natural contraction.
		
		Consider the morphism
		\begin{align*}
		\phi_M :\Sigma'\to \mathbb{P}^{h^0(\Sigma', M)-1}
		\end{align*}
		defined by the complete linear system $|M|$. If $\dim \phi_M(\Sigma')=1$, as we have $h^0(\Sigma, \rounddown{L}-F)=h^0(\Sigma', M-F')>0$, we know that the general fiber of the pencil induced by $\phi_M$ is nothing but $F'$. Thus $ M\equiv aF'$, where $a \ge h^0(\Sigma', M)-1$. It follows that
		$$
		L^3 \ge ((\pi^*L)^2\cdot M)=a (2\sigma^*K_F)^2=4a\ge 4h^0(\Sigma', M)-4.
		$$
		
		Now assume that $\dim \phi_M(\Sigma')=2$. Denote by $C'$ a general fiber of $\phi_M$. Then $M^2 \equiv a C'$ as a numerical equivalence of $1$-cycles, where $a \ge h^0(\Sigma', M)-2$.
		It follows that
		$$
		L^3 \ge ((\pi^*L)\cdot M^2) = a \left((\pi^*L) \cdot C'\right) \ge (h^0(\Sigma', M)-2) \left( (\pi^*L) \cdot C' \right).
		$$
		On the other hand, notice that $C'$ lies in a general fiber $F'$. Let $C = \sigma(C')$. Since $C$ gives a rational pencil on $F$, we have $h^0(F, C) \ge 2$. By Fact $2$, $(K_F\cdot C) \ge 2$, which implies that $((\pi^*L) \cdot C') \ge 4$. Thus it follows that
		$$
		L^3\ge 4h^0(\Sigma, \rounddown{L})-8.
		$$
		
		Finally, we consider the case when $\dim \phi_M(\Sigma')=3$. In this case, $\phi_M|_{F'}$ is generically finite. Since
		$h^0(F, 2K_F)=3$ and $M|_{F'} \le \sigma^*(2K_F)$, this forces $M|_{F'} \sim \sigma^*(2K_F)$. By Fact 1, $|2K_F|$ induces a generically finite morphism of degree $4$. This implies that $\deg \phi_M \ge 4$. Thus
		$$
		L^3 \ge M^3 \ge ( \deg \phi_M) (h^0(\Sigma', M)-3) \ge 4h^0(\Sigma, \rounddown{L})-12.
		$$
		The proof is completed.
	\end{proof}
	
	\subsection{Proof of Proposition \ref{prop:(1,1) surface}.}
	
	Let $f: X\to B$ be as in Proposition \ref{prop:(1,1) surface}. Let $\pi : Y \to X$ be a resolution of singularities of $X$ such that $\pi$ is isomorphic over the smooth locus of $X$. Let $\tilde{f}: Y\to B$ be the induced fibration. Denote by $A$ an ample  and effective $\QQ$-divisor on $B$. We extend $Y \to X\to B$ to an integral model $\mathcal{Y}\to\mathcal{X}\to\mathcal{B}\to \CZ$ and further extend $A$ to $\mathcal{A}$ exactly the same as in \S \ref{interal model} so that (i)--(iv) in \S \ref{interal model} all hold here. 
	
	Let $\tilde{f}: Y_0\to X_0\to B_0$ be a general closed fiber of $\mathcal{Y}\to \mathcal{X}\to \mathcal{B}$ defined over a field of positive characteristic. We still use $\pi$ to denote the morphism $Y_0\to X_0$. Apply Lemma \ref{lem:(1,1) surface1} to the $\QQ$-divisor $2(\pi^*K_{X_0/B_0}+\tilde{f}^*A_0)$ and follow the same strategy as in \S \ref{modpproof}. Then we deduce that
	$$
	8(\pi^*K_{X_0/B_0}+\tilde{f}^*A_0)^3 \ge 4 \deg f_* \omega_{X_0/B_0}^{[2]}.
	$$
	Go back to characteristic zero and use the arbitrarity of $\deg A$. It follows that
	\begin{equation} \label{eq:(1,1)}
	2K_{X/B}^3\ge\mathrm{deg}f_{*}\omega_{X/B}^{[2]}.
	\end{equation}
	
	Since $F$ is a (1,1)-surface, we have $R^1f_{*}\omega_X=0$ and $R^2f_{*}\omega_X=\omega_B$. Furthermore, $f_{*}\omega_{X/B}$ is a semi-positive line bundle. Thus we have $f_{*}\omega_{X/B}=\mathcal{O}_B(D)$ for a certain divisor $D$ on $B$. Thus we have
	\begin{eqnarray*}
		\chi(X, \omega_X) &= & \chi(B, f_{*}\omega_X)+\chi(B, \omega_B)\\
		& = & \chi(B, \CO_B(K_B+D)) + b - 1\\
		& = & \deg D+2(b-1).
	\end{eqnarray*}
	Here we recall a result of Ohno \cite[Lemma 2.8]{Ohno} which says that
	$$
	\deg f_* \omega_{X/B}^{[2]} = \frac{1}{2}K_{X/B}^3 + 3 \chi(X, \omega_X) - 3 \chi(F, \CO_F)(b-1) +l(2),
	$$
	where $l(2) \ge 0$ is the correction term in Reid's Riemann-Roch formula \cite{Reid}. The above two equalities together yield
	$$
	\deg f_* \omega_{X/B}^{[2]} = \frac{1}{2}K_{X/B}^3  + 3 \deg f_* \omega_{X/B} + l(2).
	$$
	Combining this with \eqref{eq:(1,1)}, we have
	$$
	K_{X/B}^3 \ge 2 \deg f_* \omega_{X/B} + \frac{2}{3} l(2) \ge 2 \deg f_* \omega_{X/B} .
	$$
	The whole proof is completed.

	\section{The remaining case} \label{remainingproof}
	Our main result in this section is the following.
	
	\begin{prop} \label{prop:pg3}
		Let $f: X \to B$ be as in \S \ref{setting}. Suppose that $p_g(F) \ge 3$. Then we have
		$$
		K_{X/B}^3 \ge \alpha \deg f_* \omega_{X/B},
		$$
		where $\alpha = \max \{2, \frac{4K^2_F}{K^2_F+4}\}$
	\end{prop}
	
	It is obvious that $\frac{4K^2_F}{K^2_F+4} \ge 2$ once $K_F^2 \ge 4$. The strategy of the proof here is similar to previous ones, and we still need the mod $p$ reduction to get the desired lower bound.

	\subsection{Positive characteristic results}
	From previous sections, we have seen that results in positive characteristics with the desired slope and some error terms, such as Lemma \ref{lemmacharp} and Lemma \ref{lem:(1,1) surface1}, are crucial. Fortunately, when the general fiber of $f$ is not a $(1, 2)$-surface, some results of this type are known.
	
	Suppose that $k$ is a field of characteristic $p > 0$. Let $f: \Sigma \to B$ be a fibration defined over $k$ from a smooth $3$-fold $\Sigma$ to a smooth curve $B$ with general fiber $F$ a minimal surface of general type. Let $L$ be a nef $\QQ$-divisor on $\Sigma$ such that $L|_F \sim K_F$. Notice that this assumption implies that the fractional part of $L$ must be vertical with respect to $f$.
	
	\begin{lemma} \label{lemma:pg31}
		Suppose that $F$ is not a $(1, 2)$-surface. Then
		$$
		h^0(\Sigma, \rounddown{L}) \le \left(\frac{1}{4} + \frac{1}{K_F^2}\right) L^3 + \frac{K_F^2 + 4}{2}.
		$$
	\end{lemma}
	
	\begin{proof}
		The proof is identical to that of \cite[Theorem 7.1]{Zhang}. In fact, the proof \cite[Theorem 7.1]{Zhang} applies verbatim to any nef $\QQ$-divisor $L$ whose restriction on the general fiber coincides with $K_F$, not necessarily $K_{X/B}$. Moreover, by \cite[Remark 7.2]{Zhang}, the whole argument therein works regardless of the characteristic of the base field. Thus we leave the detailed proof to the interested reader.
	\end{proof}
	
	\begin{lemma} \label{lemma:pg32}
		Suppose that $p_g(F) = 3$ and the canonical map of $F$ is generically finite. Then
		$$
		L^3 \ge 2h^0(\Sigma, \rounddown{L}) - 6.
		$$
	\end{lemma}
	
	\begin{proof}
		Here the proof is similar to that of Lemma \ref{lem:(1,1) surface1}. Thus we just give a sketch.
		
		To prove this result, we may assume that $h^0(\Sigma, \rounddown{L}) \ge 4$. From the following exact sequence
		$$
		0 \to H^0(\Sigma, \rounddown{L}-F) \to H^0(\Sigma, \rounddown{L}) \to H^0(F, K_F),
		$$
		we know that $h^0(\Sigma, \rounddown{L}-F) > 0$. Choose a blow-up $\pi: \Sigma' \to \Sigma$ such that the movable part $|M|$ of $\pi^*|\rounddown{L}|$ is base point free. Denote by $F' = \pi^*F$.
		
		Consider the morphism
		$$
		\phi_M: \Sigma' \to \PP^{h^0(\Sigma', M)-1}
		$$
		defined by the complete linear system $|M|$. If $\dim \phi_M(\Sigma') = 1$, similarly to the proof of Lemma \ref{lem:(1,1) surface1}, we know that the general fiber of the pencil induced by $\phi_M$ is just $F'$. Thus $M \equiv aF'$, where $a \ge h^0(\Sigma', M) - 1$. It follows that
		$$
		L^3 \ge ((\pi^*L)^2\cdot M) = a (\sigma^*K_F)^2 = 2a \ge 2h^0(\Sigma', M) - 2.
		$$
		
		If $\dim \phi_M(\Sigma') = 2$, then using the same argument as in the proof of Lemma \ref{lem:(1,1) surface1}, we deduce that
		$$
		L^3 \ge ((\pi^*L)\cdot M^2) = a \left((\pi^*L) \cdot C'\right) \ge \left(h^0(\Sigma', M)-2\right) \left((\pi^*L) \cdot C'\right),
		$$
		where $C'$ is a general closed fiber of $\phi_M$.
		
		Again, $C'$ lies in a general fiber $F'$. Thus $((\pi^*L) \cdot C') = ((\pi^*K_F) \cdot C') = (K_F\cdot C)$, where $C = \pi(C')$. Since $C^2 \ge 0$, $p_a(C) \ge 2$ and $K_F^2 \ge 2$, by the Hodge index theorem and the adjunction formula, $(K_F\cdot C) \ge 2$. As a result,
		$$
		L^3 \ge \left((\pi^*L)|_S \cdot M|_S \right) = a \left((\pi^*L) \cdot C'\right) = 2 \left(h^0(\Sigma', M) - 2\right).
		$$
		
		If $\dim \phi_M(\Sigma') = 3$, then $\dim \phi_M(F') = 2$. Since $L|_F \sim K_F$ and the canonical map of $F$ is generically finite but not birational, we deduce that $\deg (\phi_M|_{F'}) \ge 2$. In particular, $\deg \phi_M \ge 2$. Thus
		$$
		L^3 \ge M^3 \ge (\deg \phi_M) \left(h^0(\Sigma', M) - 3\right) \ge 2 \left(h^0(\Sigma', M) - 3\right).
		$$
		The proof is now completed.
	\end{proof}

	\begin{remark}
		Notice that here we do not require that $F$ is specialized from characteristic zero, because we do not need any specific geometry of $F$ which only holds in characteristic zero. 
	\end{remark}
	
	\subsection{Proof of Proposition \ref{prop:pg3}} The proof here is similar to that of Proposition \ref{prop:(1,2) surf b2} but with a slight modification. Thus we give a sketch and leave the detail to the interested reader.
	
	Let $f: X \to B$ be as in Proposition \ref{prop:pg3} with $p_g(F) \ge 3$. Let $\pi: Y \to X$ be a resolution of singularities of $X$, and let $\tilde{f}: Y \to B$ be the induced fibration. As is in \S \ref{interal model}, we may assume that $Y \to X \to B$ extends to an integral model $\CY \to \CX \to \CB \to \CZ$ for a scheme $\CZ$ of finite type over $\ZZ$. Let $A$ be the same ample and effective $\QQ$-divisor as in \S \ref{interal model} which also extends to a relatively ample and effective $\QQ$-divisor $\CA$ on $\CB$. Notice that $\pi^*K_{X/B} + \tilde{f}^*A$ is also semi-ample.
	
	In the following, we adopt almost the same notation as in \S \ref{modpproof}. Let $\tilde{f}: Y_0 \to X_0 \to B_0$ be a general closed fiber of $\CY \to \CX \to \CB$ defined over a field of positive characteristic. We still use $\pi$ to denote the morphism $Y_0 \to X_0$. Let $L = \pi^*K_{X_0/B_0} + \tilde{f}^*A_0$. Apply Lemma \ref{lemma:pg31} and Lemma \ref{lemma:pg32}, and follow exactly the same argument as in \S \ref{modpproof}. We deduce that
	$$
	(\pi^*K_{X_0/B_0} + \tilde{f}^*A_0)^3 \ge \frac{4K_F^2}{4+K_F^2} \deg f_* \omega_{X_0/B_0}
	$$
	whenever $K_F^2 > 1$ and
	$$
	(\pi^*K_{X_0/B_0} + \tilde{f}^*A_0)^3  \ge 2 \deg f_* \omega_{X_0/B_0}
	$$
	when $p_g(F) = 3$ and the canonical map is not composed with a pencil. These two inequalities implies that
	$$
	(K_{X/B} + f^*A)^3 \ge \max \left\{2, \frac{4K^2_F}{K_F^2 + 4}\right\} \deg f_* \omega_{X/B}
	$$
	over characteristic zero. Since $\deg A > 0$ is arbitrary, the proof is completed similar as before.
	
	\subsection{An example} \label{(2,3) example}
	In this subsection, we present an example of a fibration $f$ showing that the inequality in Proposition \ref{prop:pg3} is sharp.
	
	Let $g\ge 0$  be a given integer. Let $Z=\PP^2 \times B$, where $B$ is a smooth curve of genus $g$. Let $p_1$ and $p_2$ be the two natural projections of $Z$. We choose an ample divisor $D_1$ on $\mathbb{P}^2$ and an ample divisor $D_2$ on $B$. We take $D_1$ to be a smooth curve of degree $4$ on $\PP^2$. Then the divisor
	$$
	D=p_1^* D_1 + p_2^* D_2
	$$
	is ample on $Z$. We may further assume that the linear system $|2D|$ is base point free. Thus we can choose a smooth divisor $H \in |2D|$. This induces a double cover $\pi: X \to Z$ branched along $H$. Denote by $f: X\rightarrow B$ the induced fibration from $X$ to $B$. On a general fiber $F$ of $f$, we see that $(p_1 \circ \pi)|_F:F \to \PP^2$ is a double cover over $\mathbb{P}^2$ branched along a smooth curve of degree $8$. Thus $K_F$ is ample, and it is easy to compute that $(K_F^2, p_g(F))=(2,3)$.
	
	From the construction, we know that
	$$
	K_{X/B}=\pi^*(K_{Z/B}+D)=\pi^*\left(p_1^*(K_{\mathbb{P}^2}+D_1)+p_2^*D_2 \right),
	$$
	which is an ample divisor. Thus we have
	$$
	K_{X/B}^3 = 2\left(p_1^*(K_{\PP^2}+D_1)+p_2^*D_2\right)^3=6 \deg D_2.
	$$
	By the double cover formula, we get
	\begin{align*}
	\pi_{*}\omega_{X/B}=\omega_{Z/B} \oplus \CO_Z \left( p_1^*(K_{\mathbb{P}^2}+D_1)+p_2^*D_2 \right).
	\end{align*}
	Applying the projection formula and the fact that $Z$ is a $\PP^2$-bundle over $B$, we have
	$$
	f_*\omega_{X/B} = {p_2}_* (\pi_* \omega_{X/B})=\CO_B (D_2)^{\oplus h^0(\PP^2, K_{\mathbb{P}^2}+D_1)} = \CO_B(D_2)^{\oplus 3}.
	$$
	In particular, $\deg f_*\omega_{X/B} = 3 \deg D_2$ and we have $K_{X/B}^3 = 2 \deg f_* \omega_{X/B}$. Notice that in this example, $\deg f_* \omega_{X/B}>0$ and $f$ is not isotrivial.
	
	Moreover, if we assume that $g(B) = 1$, then we have $K_X^3 = 2 \deg f_* \omega_X$. Notice that in this case, $h^1(F, \CO_F) = 0$. Thus $R^1f_* \omega_X = 0$. By \cite[Lemma 2.4]{Ohno}, $\deg f_* \omega_X = \chi(X, \omega_X)$. As a result, we obtain an irregular $3$-fold $X$ of general type with $K_X^3 = 2\chi(X, \omega_X)$.
	
	\subsection{Theorem \ref{main4} in positive characteristics} 
	
	A natural question is whether we have a slope inequality for $3$-fold fibration over curve in positive characteristics. Here we remark that \eqref{slope3} actually holds regardless of the field characteristic. We have the following theorem.
	
	\begin{theorem} \label{slopecharp}
		Let $f: X \to B$ be a fibration defined over an algebraically closed field $k$ of characteristic $p > 0$ from a $3$-fold $X$ to a smooth curve $B$ with general fiber $F$ a smooth surface of general type. Suppose that $X$ has at worst terminal singularities and $K_{X/B}$ is nef. Then
		$$
		K_{X/B}^3 \ge \left(\frac{4K_F^2}{K_F^2 + 4}\right) \deg f_* \omega_{X/B}.
		$$
	\end{theorem}
	
	\begin{proof}
		Applying Lemma \ref{lemma:pg31} for the Frobenius pullback of $K_{X/B}$ and the previous limiting argument, if $F$ is not a $(1, 2)$-surface, then Theorem \ref{slopecharp} follows in exactly the same way as before. Notice that it is even simpler in the current setting, because we no longer need to perturb $K_{X/B}$ by a divisor $A$ on $B$.
		
		We are left the case when $F$ is a $(1, 2)$-surface. In this case, we can use Lemma \ref{lemma:charp (1,2)} (which we will prove in the following) to proceed the proof using the same Frobenius limiting argument, and we obtain that
		$$
		K_{X/B}^3 \ge \deg f_* \omega_{X/B}.
		$$ 
		Since the proof is just identical as before, we leave it to the interested readers.
	\end{proof}
	
	\begin{lemma} \label{lemma:charp (1,2)}
		Let $f: \Sigma \to B$ be a fibration defined over an algebraically closed field $k$ from a smooth $3$-fold $\Sigma$ to a smooth curve $B$ with general fiber $F$ a $(1, 2)$-surface. Let $L$ be a nef and big $\QQ$-divisor on $\Sigma$ such that $L|_F \sim K_F$. Then we have
		$$
		L^3 \ge h^0(\Sigma, \rounddown{L}) - 2.
		$$
	\end{lemma}
	
	\begin{proof}
		The proof is very similar to that of Lemma \ref{lem:(1,1) surface1} and \ref{lemma:pg32}. We just give a sketch here.
		
		We may assume that $h^0(\Sigma, \rounddown{L}) \ge 3$. Since $L|_F \sim K_F$ and $p_g(F) = 2$, we deduce that $h^0(\Sigma, \rounddown{L}-F) > 0$. Choose a blow-up $\pi: \Sigma' \to \Sigma$ such that the movable part $|M|$ of $\pi^*|\rounddown{L}|$ is base point free. Denote $F' = \pi^*F$.
		
		Consider the morphism
		$$
		\phi_M: \Sigma' \to \PP^{h^0(\Sigma', M) - 1}.
		$$
		Here $1 \le \dim \phi_M(\Sigma) \le 2$. If $\dim \phi_M(\Sigma) = 1$, we just apply the proof of Lemma \ref{lemma:pg32} for this case verbatim to deduce that $M \equiv aF'$ for some integer $a \ge h^0(\Sigma', M) - 1$ and thus
		$$
		L^3 \ge aK_F^2 = h^0(\Sigma', M) - 1.
		$$
		If $\dim \phi_M(\Sigma) = 2$, then the proof of Lemma \ref{lemma:pg32} for this case also applies here, and we deduce that
		$$
		L^3 \ge \left(h^0(\Sigma, M) - 2\right)(K_F\cdot C),
		$$
		where $C$ is the image of a general fiber of $\phi_M$ under $\pi$. Since $C$ gives a family of curves on $F$, we deduce that $(K_F\cdot C) \ge 1$. Thus he proof is completed.
	\end{proof}
	
	\begin{remark}
		Similarly to the characteristic zero case, it is an interesting question whether the $f$-nefness of $K_X$ implies the nefness of $K_{X/B}$ in characteristic $p > 0$. In fact, under the setting of Theorem \ref{slopecharp}, Patakfalvi \cite[Theorem 1.1]{Patakfalvi} has given a positive answer to this question when $(p, r) = 1$, where $r$ is the index of $K_X$.
	\end{remark}

\section{A Noether-Severi type inequality}

Our main result in this section is the following.

\begin{theorem} [Theorem \ref{main5}]
	Let $X$ be an irregular and minimal $3$-fold of general type over an algebraically closed field of characteristic zero. Suppose that the Albanese fiber of $X$ is not a $(1, 2)$-surface. Then we have the following sharp inequality:
	$$
	K^3_X \ge 2 \chi(X, \omega_X).
	$$
	Moreover, if the equality holds, then $h^1(X, \CO_X) = 1$.
\end{theorem}

\begin{proof}
	If the Albanese dimension of $X$ is at least two, then we have already known that (e.g. \cite[Theorem 1.1]{Zhang_Severi} and \cite[\S 6]{Zhang_small_volume})
	$$
	K_X^3 \ge 4 \chi(X, \omega_X).
	$$
	In the following, we may assume that the Albanese dimension of $X$ is one.
	Let 
	$$
	f: X \to B
	$$
	be the Stein factorization of the Albanese map of $X$. Thus $g(B) \ge h^1(X, \CO_X) > 0$, and a general fiber $F$ of $f$ is a smooth minimal surface of general type, not a $(1, 2)$-surface by our assumption. Clearly, $f$ is relatively minimal. By Theorem \ref{main3}, we know that
	$$
	K_{X/B}^3 \ge 2 \deg f_* \omega_{X/B}.
	$$
	On the other hand, by \cite[Lemma 2.4, 2.5]{Ohno},
	$$
	\deg f_* \omega_{X/B} \ge \chi(X, \omega_X) - \chi(B, \omega_B) \chi(F, \omega_F).
	$$
	In the meantime, since $K_{X/B} = K_X - f^*K_B$, we have
	$$
	K_{X/B}^3 = K_X^3 - 6K_F^2(g(B)-1) = K_X^3 - 6K_F^2 \chi(B, \omega_B).
	$$
	Combine all above together. It follows that
	$$
	K_X^3 \ge 2\chi(X, \omega_X) + 2\left(3K_F^2 - \chi(F, \omega_F) \right) \chi(B, \omega_B).
	$$
	It is easy to check that $3K_F^2 - \chi(F, \omega_F) > 0$ once $F$ is not a $(1, 2)$-surface. Thus 
	$$
	K_X^3 \ge 2\chi(X, \omega_X).
	$$
	The sharpness of this inequality has been shown in \S \ref{(2,3) example}. Moreover, if the equality holds, we have $\chi(B, \omega_B) = 0$. Thus $g(B) = 1$, which forces $h^1(X, \CO_X) = 1$. 
	
	The whole proof is completed.
\end{proof}

	\bibliography{slope3fold}
	\bibliographystyle{amsplain}
	
\end{document}